\documentclass[a4paper,oneside,english]{amsart}
\usepackage[T1]{fontenc}
\usepackage[latin9]{inputenc}
\usepackage{amsthm}
\usepackage{amstext}
\usepackage{amssymb}

\makeatletter
\numberwithin{equation}{section}
\numberwithin{figure}{section}

\usepackage{amsthm}

\numberwithin{equation}{section} 
\numberwithin{figure}{section} 
\theoremstyle{plain}
\newtheorem{thm}{Theorem}[section]
  \theoremstyle{definition}
  \newtheorem{defn}[thm]{Definition}
  \theoremstyle{plain}
  \newtheorem{lem}[thm]{Lemma}
 \theoremstyle{definition}
  \newtheorem{example}[thm]{Example}
  \theoremstyle{remark}
  \newtheorem{rem}[thm]{Remark}
  \theoremstyle{plain}
  \newtheorem{prop}[thm]{Proposition}
  \theoremstyle{plain}
  \newtheorem{cor}[thm]{Corollary}

\usepackage{dsfont}
\usepackage{mathrsfs}
\usepackage{eucal}

\makeatother

\usepackage{babel}

\begin{document}

\title{Stability of Anosov Hamiltonian Structures}

\author{Will J. Merry and Gabriel P. Paternain}
\begin{abstract}
Let $(M^{n},g)$ denote a closed Riemannian manifold ($n\geq3$) which
admits a metric of negative curvature (not necessarily equal to $g$).
Let $\omega_{1}:=\omega_{0}+\pi^{*}\sigma$ denote a twisted symplectic
form on $TM$, where $\sigma\in\Omega^{2}\left(M\right)$ is a closed
$2$-form and $\omega_{0}$ is the symplectic structure on $TM$ obtained
by pulling back the canonical symplectic form $dx\wedge dp$ on $T^{*}M$
via the Riemannian metric. Let $\Sigma_{k}$ be the hypersurface $|v|=\sqrt{2k}$.
We prove that if $n$ is odd and the Hamiltonian structure $(\Sigma_{k},\omega_{1})$
is Anosov with $C^{1}$ weak bundles then $(\Sigma_{k},\omega_{1})$
is stable if and only if it is contact. If $n$ is even and in addition
the Hamiltonian structure is $1/2$-pinched, then the same conclusion
holds. As a corollary we deduce that if $g$ is negatively curved,
strictly $1/4$-pinched and $\sigma$ is not exact then the Hamiltonian
structure $(\Sigma_{k},\omega_{1})$ is never stable for all sufficiently
large $k$. 
\end{abstract}

\address{Department of Pure Mathematics and Mathematical Statistics, University
of Cambridge, Cambridge CB3 0WB, England}

\email{\texttt{w.merry@dpmms.cam.ac.uk, g.p.paternain@dpmms.cam.ac.uk}}

\maketitle

\section{Introduction}

Let $\Sigma$ be a closed oriented manifold of dimension $2n-1$.
A \emph{Hamiltonian structure} on $\Sigma$ is a closed $2$-form
$\omega$ such that $\omega^{n-1}\ne0$. Its kernel $\ker\,\omega$
defines an orientable $1$-dimensional foliation.

A natural condition to impose on a Hamiltonian structure is \emph{stability};
this asserts the existence of a $1$-form $\lambda$ such that $\ker\,\omega\subseteq\ker\, d\lambda$
and such that $\lambda\wedge\omega^{n-1}>0$. The $1$-form $\lambda$
is known as a \emph{stabilizing $1$-form}.

A stronger condition one might like to impose is the following: $\left(\Sigma,\omega\right)$
is called of \emph{contact type} if we can find a $1$-form $\lambda$
such that $d\lambda=\omega$ and $\lambda\wedge\omega^{n-1}>0$. In
particular $\lambda$ is a stabilizing $1$-form and $\lambda$ is
a contact form on $\Sigma$. Note that $\left(\Sigma,\omega\right)$
can be of contact type only if $\omega$ is exact. The stability condition
first appeared in \cite{HoferZehnder1994} as a condition for which
the Weinstein conjecture could be proved. More recently, stability
has been recognized as a key condition to produce compactness results
in Symplectic Field Theory \cite{BourgeoisEliashbergHoferWysockiZehnder2003,CieliebakMohnke2005,EliashbergKimPolterovich2006}
and Rabinowitz Floer homology \cite{CieliebakFrauenfelderPaternain2009}.
This paper is motivated by the desire to generalize a result in the
latter reference, as we explain below.

Let $F$ be any vector field spanning $\ker\,\omega$. We say that
$\left(\Sigma,\omega\right)$ is an \emph{Anosov} Hamiltonian structure
if the flow $\phi_{t}$ of $F$ is Anosov. Recall that this asserts
the existence of a $d\phi_{t}$-invariant splitting \[
T\Sigma=\mathbb{R}F\oplus E^{s}\oplus E^{u},\]
 where $\mathbb{R}F$ is the $1$-dimensional distribution spanned
by $F$, and such that there exist constants $C,\mu>0$ such that
for all $x\in\Sigma$ and $t\geq0$,\[
\left|d_{x}\phi_{t}(\xi)\right|\leq C\left|\xi\right|e^{-\mu t}\ \mbox{for }\xi\in E^{s}(x);\]
 \[
\left|d_{x}\phi_{-t}\left(\xi\right)\right|\leq C\left|\xi\right|e^{-\mu t}\ \mbox{for }\xi\in E^{u}(x).\]
 The Anosov condition is invariant under time changes, and so is independent
of the choice of vector field $F$. In other words, it is intrinsic
to the Hamiltonian structure $(\Sigma,\omega)$. The \textit{weak
bundles} $E^{+}:=\mathbb{R}F\oplus E^{s}$ and $E^{-}:=\mathbb{R}F\oplus E^{u}$
are also invariant under time changes. \\

We say that $\phi_{t}:\Sigma\rightarrow\Sigma$ is $1/2$-\textit{pinched}
(or $1$\emph{-bunched} \cite{Hasselblatt1994a}) if there exist positive
constants $C,A,a$ with $A<2a$ such that

\[
\frac{1}{C}|\xi|e^{-At}\leq|d_{x}\phi_{t}(\xi)|\leq C|\xi|e^{-at}\;\;\mbox{{\rm for}}\;\xi\in E^{s}\;\;\mbox{{\rm and}}\; t\geq0,\]
 \[
\frac{1}{C}|\xi|e^{-At}\leq|d_{x}\phi_{-t}(\xi)|\leq C|\xi|e^{-at}\;\;\mbox{{\rm for}}\;\xi\in E^{u}\;\;\mbox{{\rm and}}\; t\geq0.\]

We say that an Anosov Hamiltonian structure satisfies the $1/2$-pinching
condition if there exists \emph{some} vector field $F$ spanning $\ker\,\omega$
whose flow satisfies the $1/2$-pinching condition.\\

Here is the situation we are actually interested in. Let $(M,g)$
be a closed Riemannian manifold that admits a background metric of
negative curvature (possibly different from $g$) and $\pi:TM\rightarrow M$
the tangent bundle. Throughout the paper we let $\omega_{0}$ denote
the symplectic form on $TM$ obtained by pulling back the canonical
symplectic form $dx\wedge dp$ on $T^{*}M$ via the Riemannian metric.
The form $\omega_{0}$ is exact; if $\alpha\in\Omega^{1}(TM)$ denotes
the $1$-form defined by \begin{equation}
\alpha_{v}(\xi)=\left\langle d_{v}\pi(\xi),v\right\rangle ,\label{eq:alpha}\end{equation}
 then it is well known that $\omega_{0}=-d\alpha$. Suppose $\sigma\in\Omega^{2}\left(M\right)$
is a closed $2$-form on $M$. Given $\varepsilon\in\mathbb{R}$ we
define \[
\omega_{\varepsilon}:=\omega+\varepsilon\pi^{*}\sigma.\]
 Let $F_{\varepsilon}$ denote the symplectic gradient of the Hamiltonian
\[
H(x,v)=\frac{1}{2}\left|v\right|^{2}\]
 with respect to $\omega_{\varepsilon}$ and let $\phi_{t}^{\varepsilon}$
denote the flow of $F_{\varepsilon}$ with respect to $\omega_{\varepsilon}$.
Note that $\phi_{t}^{0}$ is simply the geodesic flow. This flow models
the motion of a particle of unit mass and charge $\varepsilon$ under
the effect of a magnetic field, whose \emph{Lorentz force} $Y:TM\rightarrow TM$
is the bundle map uniquely determined by \begin{equation}
\sigma_{x}(u,v)=\left\langle Y_{x}(u),v\right\rangle \label{eq:lorentz}\end{equation}
 for all $u,v\in T_{x}M$ and all $x\in M$.

The family $\left\{ \omega_{\varepsilon}\right\} $ for $\varepsilon\in\left[0,1\right]$
interpolates between the standard symplectic form $\omega_{0}$ and
the form $\omega_{1}$. The form $\omega_{1}$ is called a \emph{twisted
symplectic structure} \cite{ArnoldGivental1990} and the flow $\phi_{t}^{1}$
is called a \emph{twisted geodesic flow} or a \emph{magnetic flow}.

Let $\Sigma_{k}=H^{-1}(k)$. We are interested in Anosov Hamiltonian
structures of the form $(\Sigma_{k},\omega_{1})$. \\

Here is the main result we present. Let $n=\dim\, M$. \\

\noindent \textbf{Theorem A.} \emph{Suppose $(\Sigma_{k},\omega_{1})$
is an Anosov Hamiltonian structure and $n\geq3$. Assume in addition:}
\begin{itemize}
\item \emph{If $n$ is odd, $(\Sigma_{k},\omega_{1})$ has weak bundles
of class $C^{1}$;} 
\item \emph{If $n$ is even, $(\Sigma_{k},\omega_{1})$ is $1/2$-pinched.}
\end{itemize}
\emph{Then $(\Sigma_{k},\omega_{1})$ is stable if and only if it
is of contact type. In particular, if $(\Sigma_{k},\omega_{1})$ is
stable, then $\sigma$ must be exact.}\\

The last statement in the theorem can be seen as follows. Since $n=\dim\, M\geq3$,
the Gysin sequence of the sphere bundle $\pi|\Sigma_{k}:\Sigma_{k}\rightarrow M$
shows that $(\pi|\Sigma_{k})^{*}:H^{2}(M,\mathbb{R})\rightarrow H^{2}(\Sigma_{k},\mathbb{R})$
is an isomorphism for $n\geq4$ and injective for $n=3$. Since $\omega_{1}=-d\alpha+\pi^{*}\sigma$,
the assertion that $\omega_{1}$ is exact on $\Sigma_{k}$ implies
that $\pi^{*}\sigma|\Sigma_{k}$ is exact. Putting this together we
conclude that $\sigma$ is exact.

The bunching condition is a necessary one in the even dimensional
case. Indeed, consider the twisted geodesic flow $\phi_{t}^{1}$ on
compact quotients of complex hyperbolic space with $\sigma$ given
by the K\"ahler form. Then for $k$ sufficiently large, $\phi_{t}^{1}|\Sigma_{k}$
is Anosov, and $\Sigma_{k}$ is stable but not contact ($\sigma$
is not exact). The flow $\phi_{t}^{1}$ is algebraic and thus the
stable and unstable bundles are real analytic. A stabilizing 1-form
$\lambda$ can be defined by setting $\lambda(F_{1})=1$ and $\ker\,\lambda=E_{1}^{s}\oplus E_{1}^{u}$.
The flow $\phi_{t}^{1}$ is not $1/2$-pinched since it has 2:1 resonances.
It seems a reasonable conjecture that these are in fact the only cases
where $C^{1}$ weak bundles is not sufficient to ensure that the conclusion
of Theorem A holds. A well known theorem \cite{HirschPughShub1977}
states that the $1/2$-pinching condition implies that the weak bundles
are of class $C^{1}$. However the pinching condition is strictly
stronger than requiring the weak bundles to be of class $C^{1}$ as
the example of complex hyperbolic space described above shows.\\

We will prove in Proposition \ref{pro:ANOSOV 12PINCHED} that if $g$
is negatively curved and strictly $1/4$-pinched, for $k$ sufficiently
large, the flow $\phi_{t}^{1}:\Sigma_{k}\rightarrow\Sigma_{k}$ is
Anosov and $1/2$-pinched. Thus, as a corollary of Theorem A we obtain
the following.\\

\noindent \textbf{Corollary B.} \emph{Suppose $n\geq3$ and $g$ is
negatively curved and strictly $1/4$-pinched. Then for any $k$ sufficiently
large, the hypersurface $\Sigma_{k}\subset TM$ is} \textbf{\emph{not}}
\emph{stable if $\sigma$ is not exact.}\\

The corollary was first proved in \cite[Theorem 1.4]{CieliebakFrauenfelderPaternain2009}
for the case of $n$ even (but not Theorem A) and this previous result
was the motivation for the present paper. It shows that the stability
condition may fail for whole intervals of energy levels.

A caveat about the word {}``stable'' is in order. One of the most
remarkable features about Anosov systems is that they are structurally
stable. This means that nearby systems are orbitally equivalent via
a \textit{homeomorphism} which in general is \textit{not} $C^{1}$.
The stability condition in Symplectic Geometry is equivalent to the
existence of a thickening of the hypersurface with smoothly conjugate
characteristic foliations \cite[Lemma 2.3]{CieliebakMohnke2005}.
In some sense it is this additional smoothness that is being exploited
in Theorem A. One can put this into different words as follows: The
existence of the form 1-form $\lambda$ means that one can find a
parametrization of the characteristic foliation such that the hyperplane
bundle $E^{s}\oplus E^{u}$ is smooth. In general, $E^{s}\oplus E^{u}$
is only H\"older continuous.

The assumption that $M$ admits a background metric of negative curvature
is most likely superfluous. We use it to construct a conjugacy between
our flow and the geodesic flow of a metric of negative curvature on
$M$. We use this conjugacy to show that the fundamental group $\pi_{1}(M)$
acts as a `North-South dynamics' (see Section \ref{sec:North-South-Dynamics}
and in particular Theorem \ref{thm:north south dynamic}) on the space
of stable leaves on the universal covering, and that the space of
leaves admit a `flip map' (see the discussion before the proof of
Theorem \ref{strong}). In order to remove the assumption of negative
curvature one would need to prove Theorem \ref{thm:north south dynamic}
and Theorem \ref{strong} directly instead of constructing such a
conjugacy.\newline

\emph{Acknowlegment}: We thank the referee for Remark \ref{Referee remark}.

\section{\label{sec:Preliminaries-on-Hamiltonia}Preliminaries on Hamiltonian
structures, holonomy and the Kanai connection}

\begin{defn} A \emph{Hamiltonian structure} is a pair $\left(\Sigma,\omega\right)$
where $\Sigma^{2n-1}$ is a closed oriented manifold and $\omega$
is a closed $2$-form on $\Sigma$ such that $\omega^{n-1}\ne0$ everywhere.
\end{defn} If $\left(\Sigma,\omega\right)$ is a Hamiltonian structure
then $\ker\,\omega$ defines an orientable $1$-dimensional foliation
of $\Sigma$, which we call the \emph{characteristic foliation}. \begin{defn}
Let $\phi_{t}:N\rightarrow N$ be a smooth flow on a closed Riemannian
manifold $N$, and let $F$ denote its infinitesimal generator. We
say that $\phi_{t}$ is \emph{Anosov} if there exists a $d\phi_{t}$-invariant
splitting \[
T\Sigma=\mathbb{R}F\oplus E^{s}\oplus E^{u},\]
 where $\mathbb{R}F$ is the $1$-dimensional distribution spanned
by $F$, and such that there exist constants $C,\mu>0$ such that
for all $x\in\Sigma$ and $t\geq0$,\[
\left|d_{x}\phi_{t}(\xi)\right|\leq C\left|\xi\right|e^{-\mu t}\ \mbox{for }\xi\in E^{s}(x);\]
 \[
\left|d_{x}\phi_{-t}\left(\xi\right)\right|\leq C\left|\xi\right|e^{-\mu t}\ \mbox{for }\xi\in E^{u}(x).\]

\end{defn} The Anosov property is invariant under time changes (see
\cite[Lemma 1.2]{deLaLlaveMarcoMoryion1986} or \cite[Proposition 17.4.5]{KatokHasselblatt1995}),
and hence we can define a Hamiltonian structure $\left(\Sigma,\omega\right)$
to be \emph{Anosov} if the flow $\phi_{t}$ of \emph{any} vector field
$F$ spanning $\ker\,\omega$ is Anosov.

We say that $\phi_{t}:\Sigma\rightarrow\Sigma$ is $1/2$-\textit{pinched}
(or $1$\emph{-bunched} \cite{Hasselblatt1994a}) if there exist positive
constants $C,A,a$ with $A<2a$ such that

\[
\frac{1}{C}|\xi|e^{-At}\leq|d_{x}\phi_{t}(\xi)|\leq C|\xi|e^{-at}\;\;\mbox{{\rm for}}\;\xi\in E^{s}\;\;\mbox{{\rm and}}\; t\geq0,\]
 \[
\frac{1}{C}|\xi|e^{-At}\leq|d_{x}\phi_{-t}(\xi)|\leq C|\xi|e^{-at}\;\;\mbox{{\rm for}}\;\xi\in E^{u}\;\;\mbox{{\rm and}}\; t\geq0.\]

We say that an Anosov Hamiltonian structure satisfies the $1/2$-pinching
condition if there exists \emph{some} vector field $F$ spanning $\mbox{{\rm ker}}\,\omega$
whose flow satisfies the $1/2$-pinching condition.\\

The $1/2$-pinching condition is a natural one to study, and should
be thought of as a statement about being `strongly hyperbolic'. Unsurprisingly,
an Anosov system possessing this enhanced degree of hyperbolicity
enjoys greater regularity. More specifically, write $E^{+}$ and $E^{-}$
for the \emph{weak stable} and \emph{weak unstable} bundles $E^{s}\oplus\mathbb{R}F$
and $E^{u}\oplus\mathbb{R}F$ respectively. If an Anosov Hamiltonian
structure is $1/2$-pinched, then $E^{+}$ and $E^{-}$ are of class
$C^{1}$ \cite{HirschPughShub1977}. The next example will be crucial
for the proof of Corollary B. \begin{example} \label{exa:example sec curv}Let
$(M,g)$ be a closed manifold with negative sectional curvature $K$.
Then the geodesic flow $\phi_{t}:SM\rightarrow SM$ is Anosov (see
\cite[Section 17.6]{KatokHasselblatt1995}). In fact, we can say more.
By compactness we can find constants $k_{1}\geq k_{0}>0$ such that
\[
-k_{1}^{2}\leq K\leq-k_{0}^{2}.\]
 Then, comparison theorems show that \cite[Theorem 3.2.17]{Klingenberg1995}
(see also \cite[Proposition 3.2]{Knieper2002}) there is a constant
$C>0$ such that \begin{equation}
\frac{1}{C}|\xi|e^{-k_{1}t}\leq|d_{v}\phi_{t}(\xi)|\leq C|\xi|e^{-k_{0}t}\;\;\mbox{{\rm for}}\;\xi\in E^{s}(v)\;\;\mbox{{\rm and}}\; t\geq0,\label{eq:a1}\end{equation}
 \begin{equation}
\frac{1}{C}|\xi|e^{-k_{1}t}\leq|d_{v}\phi_{-t}(\xi)|\leq C|\xi|e^{-k_{0}t}\;\;\mbox{{\rm for}}\;\xi\in E^{u}(v)\;\;\mbox{{\rm and}}\; t\geq0.\label{eq:a2}\end{equation}
 We see that $\phi_{t}$ is $1/2$-pinched as long as $k_{1}<2k_{0}$.
Therefore the geodesic flow of a metric whose sectional curvature
satisfies $-4\leq K<-1$ is $1/2$-pinched.\end{example}

We return to two more definitions regarding Hamiltonian structures.
\begin{defn} Let $\left(\Sigma,\omega\right)$ denote a Hamiltonian
structure. We say that $\left(\Sigma,\omega\right)$ is \emph{stable}
if there exists a $1$-form $\lambda$ such that $\ker\,\omega\subseteq\ker\, d\lambda$
and $\lambda\wedge\omega^{n-1}>0$. $\lambda$ is called a \emph{stabilizing
$1$-form}. Note that if $\lambda$ is a stabilizing $1$-form and
$F$ is any vector field tangent to $\ker\,\omega$ we have $i_{F}d\lambda=0$.
If $F$ is normalized so that $\lambda(F)=1$ then we say $F$ is
the \emph{Reeb vector field} of \textbf{$\lambda$}; note that the
Reeb vector field is unique\textbf{.} \end{defn} A stronger condition
is the following. \begin{defn} Let $\left(\Sigma,\omega\right)$
denote a Hamiltonian structure. We say that $\left(\Sigma,\omega\right)$
is of \emph{contact type} if there exists a $1$-form $\lambda$ such
that $d\lambda=\omega$ and $\lambda\wedge\omega^{n-1}>0$. Note if
$\left(\Sigma,\omega\right)$ is of contact type then it is certainly
stable, and that if $\left(\Sigma,\omega\right)$ is of contact type
then $\omega$ is exact. \end{defn} Now let $\left(\Sigma,\omega\right)$
denote a stable Anosov Hamiltonian structure with weak bundles of
class $C^{1}$. Let $\lambda$ be a stabilizing $1$-form and let
$F$ be the Reeb vector field of $\lambda$. Let $\phi_{t}$ denote
the flow of $F$; then $\phi_{t}$ is Anosov and $T\Sigma=\mathbb{R}F\oplus E^{s}\oplus E^{u}$,
with $\ker\,\lambda=E^{s}\oplus E^{u}$. Then the weak bundles are
$C^{1}$; since $\lambda$ is $C^{\infty}$ (and so the bundle $\ker\,\lambda$
is of class $C^{\infty}$), it follows that the strong bundles $E^{s}$
and $E^{u}$ are also of class $C^{1}$. The importance of this is
that, as we shall see in Theorem \ref{thm:(the-Kanai-connection)}
below, under these conditions there exists a unique connection $\nabla$
on $\Sigma$ called the \emph{Kanai connection} which satisfies certain
desirable properties. This was originally introduced by Kanai in \cite{Kanai1988};
see also \cite{Kanai1993,FeresKatok1990}.

\subsection*{\label{sec:holonomy}Holonomy}

We briefly recall the concept of holonomy transport along the weak
(un)stable foliations defined by an Anosov flow.

Throughout the remainder of this section, let $N$ denote a closed
manifold of dimension $2n-1$ and $\phi_{t}:N\rightarrow N$ an Anosov
flow on $N$ with infinitesimal generator $F$. We also assume that
$E:=E^{s}\oplus E^{u}$ is smooth and admits a smooth symplectic form
$\omega$ which is $\phi_{t}$-invariant. We extend this form $\omega$
to a form defined on all of $TN$ by requiring that $i_{F}\omega=0$.
As above $\lambda$ is the 1-form defined by $\ker\,\lambda=E$ and
$\lambda(F)=1$.

It is well known that the subbundles $E^{s}$ and $E^{u}$, together
with the weak bundles $E^{+}$ and $E^{-}$, are all integrable. Namely,
given any $x\in N$, we define \[
W^{s}(x):=\left\{ y\in N\,:\,\mbox{dist}(\phi_{t}x,\phi_{t}y)\rightarrow0\mbox{ as }t\rightarrow\infty\right\} \]
 and \[
W^{u}(x):=\left\{ y\in N\,:\,\mbox{dist}(\phi_{t}x,\phi_{t}y)\rightarrow0\mbox{ as }t\rightarrow-\infty\right\} .\]
 The sets $W^{s}(x)$ and $W^{u}(x)$ are injectively immersed manifolds
called the \emph{strong (un)stable manifolds} at $x$ and satisfy\[
T_{x}W^{s}(x)=E^{s}(x),\ \ \ T_{x}W^{u}(x)=E^{u}(x).\]
 These define foliations $\mathcal{W}^{s}$ and $\mathcal{W}^{u}$
of $N$, called the \emph{strong (un)stable foliations}. We assume
throughout this section that these foliations $\mathcal{W}^{s},\mathcal{W}^{u}$
are of class $C^{1}$.

Similarly, given $x\in N$, we define \[
W^{+}(x):=\bigcup_{t\in\mathbb{R}}\phi_{t}[W^{s}(x)]=\bigcup_{t\in\mathbb{R}}W^{s}(\phi_{t}x)\]
 and \[
W^{-}(x):=\bigcup_{t\in\mathbb{R}}\phi_{t}[W^{u}(x)]=\bigcup_{t\in\mathbb{R}}W^{u}(\phi_{t}x),\]
 which are then the \emph{weak (un)stable manifolds} at $x$. They
satisfy\[
T_{x}W^{+}(x)=E^{+}(x),\ \ \ T_{x}W^{-}(x)=E^{-}(x),\]
 and define foliations $\mathcal{W}^{+}$ and $\mathcal{W}^{-}$ of
$M$.\\

Consider the foliations $\mathcal{W}^{u}$ and $\mathcal{W}^{+}$.
It is well known that these are transverse to each other (i.e. $E^{u}(x)\cap E^{+}(x)=\{0\}$)
and have complementary dimensions $n-1$ and $n$ respectively. The
same is of course true of the foliations $\mathcal{W}^{s}$ and $\mathcal{W}^{-}$,
and all of what we say below can be repeated for them.

By a \emph{foliation chart} we mean a diffeomorphism $\varphi:U\rightarrow(-1,1)^{n-1}\times(-1,1)^{n}$
of the form $x\mapsto(\varphi_{1}(x),\varphi_{2}(x))$ where $U\subseteq N$
is open, such that the connected components of $\mathcal{W}^{u}|U$
are given by $\varphi_{2}=\mbox{const}$ and the connected components
of $\mathcal{W}^{+}|U$ are given by $\varphi_{1}=\mbox{const}$.
We then call $U$ a \emph{foliated neighborhood}.

Let $ $$(\varphi,U)$ denote a foliated chart defined on $N$. Given
$x\in U$, let $W_{U}^{u}(x):=W^{u}(x)\cap U$. Suppose $y\in W_{U}^{+}(x)$
lies in the same connected component of $W_{U}^{+}(x)$ as $x$. We
want to define the \emph{holonomy map $\mathcal{H}_{x,y}^{U}:W_{U}^{u}(x)\rightarrow W_{U}^{u}(y)$
along the leaves of $\mathcal{W}^{+}$} . This is defined as follows.
Suppose $p\in W_{U}^{u}(x)$. Then there exists a unique point $q\in W^{+}(p)\cap W_{U}^{u}(y)$,
and we define $\mathcal{H}_{x,y}^{U}(p)=q$. The map $\mathcal{H}_{x,y}^{U}$
is of class $C^{r}$ for $r\geq1$ if the foliations $\mathcal{W}^{u}$
and $\mathcal{W}^{+}$ are also of class $C^{r}$.

More generally, suppose $\gamma:[0,T]\rightarrow N$ is a smooth curve
such that $\gamma(0)=x$ and $\gamma(t)\in W^{+}(x)$ for all $t$.
Let $y:=\gamma(T)\in W^{+}(x)$. Then we can define $\mathcal{H}_{x,y}^{\gamma}$
by covering the image of $\gamma$ with foliated charts $(\varphi_{i},U_{i})$
for $i=1,\dots,l$ with $x\in U_{1}$ and $y\in U_{l}$, and choosing
points $0=t_{0},\dots,t_{l}=T$ such that $\gamma(t_{i-1}),\gamma(t_{i})\in U_{i}$
for each $1\leq i\leq l$ and then setting \[
\mathcal{H}_{x,y}^{\gamma}:=\mathcal{H}_{\gamma(t_{l-1}),y}^{U_{l}}\circ\mathcal{H}_{\gamma(t_{l-2}),\gamma(t_{l-1})}^{U_{l-1}}\circ\dots\circ\mathcal{H}_{x,\gamma(t_{1})}^{U_{1}}.\]
 It can be shown that $\mathcal{H}_{x,y}^{\gamma}$ only depends on
the homotopy class of $\gamma\in W^{+}(x)$. One can check that if
$x,y,z$ are in the image of $\gamma$, then after suitably restricting
the domains of definition, it holds that\[
\mathcal{H}_{y,z}^{\gamma}\circ\mathcal{H}_{x,y}^{\gamma}=\mathcal{H}_{x,z}^{\gamma}.\]
 Moreover, since the foliations are $\phi_{t}$-invariant, for any
curve $\gamma$ in the weak unstable foliation $\mathcal{W}^{+}$
we have\begin{equation}
\phi_{t}\circ\mathcal{H}_{x,y}^{\gamma}\circ\phi_{-t}=\mathcal{H}_{\phi_{t}x,\phi_{t}y}^{\phi_{t}\circ\gamma}.\label{eq:hol eq}\end{equation}

Next, we consider the differential of $\mathcal{H}_{x,y}^{\gamma}$,
known as the \emph{holonomy transport along $\gamma$}, \[
H_{x,y}^{\gamma}(p):=d_{p}\mathcal{H}_{x,y}^{\gamma}:T_{p}W^{u}(x)\rightarrow T_{\mathcal{H}_{x,y}^{\gamma}(p)}W^{u}(y).\]
 In particular we write \[
H_{x,y}^{\gamma}:E^{u}(x)\rightarrow E^{u}(y)\]
 for the map $H_{x,y}^{\gamma}(x):T_{x}W^{u}(x)\rightarrow T_{y}W^{u}(y)$.
Note that differentiating \eqref{eq:hol eq} gives

\begin{equation}
d_{y}\phi_{t}\circ H_{x,y}^{\gamma}\circ d_{\phi_{-t}x}\phi_{-t}=H_{\phi_{t}x,\phi_{t}y}^{\phi_{t}\circ\gamma}\ \ \ \mbox{as maps }E^{u}(\phi_{t}x)\rightarrow E^{u}(\phi_{t}y).\label{eq:diff hol eq}\end{equation}
 We say that a vector field $X\in\Gamma(E^{u})$ is \emph{invariant
under the holonomy transport along the leaves of} \textbf{$\mathcal{W}^{+}$}
if for any curve $\gamma$ from $x$ to $y$ contained in $W^{+}(x)$
it holds that \[
H_{x,y}^{\gamma}(p)(X(p))=X(\mathcal{H}_{x,y}^{\gamma}(p))\ \ \ \mbox{for all\ }p\in W^{u}(x).\]

\subsection*{The Kanai connection}

We now recall the definition and main features of the Kanai connection.

Let $I$ be the $(1,1)$-tensor on $N$ given by $I(v)=-v$ for $v\in E^{s}$,
$I(v)=v$ for $v\in E^{u}$ and $I(F)=0$. Consider the symmetric
non-degenerate bilinear form given by \[
h(X,Y):=\omega(X,IY)+\lambda\otimes\lambda(X,Y).\]
 The pseudo-Riemannian metric $h$ is of class $C^{1}$ and thus there
exists a unique $C^{0}$ affine connection $\nabla$ such that: 
\begin{enumerate}
\item $h$ is parallel with respect to $\nabla$; 
\item $\nabla$ has torsion $\omega\otimes F$. 
\end{enumerate}
\begin{thm} \label{thm:(the-Kanai-connection)} The Kanai connection
$\nabla$ has the following properties: 
\begin{enumerate}
\item $\nabla$ is $\phi_{t}$-invariant, $\nabla\omega=0$, $\nabla_{F}=L_{F}$
and $\nabla F=0$. 
\item The Anosov splitting is invariant under $\nabla$, that is, if $X_{s}\in\Gamma(E^{s}),X_{u}\in\Gamma(E^{u})$
and $Y$ is any vector field on $N$ then \[
\nabla_{Y}X_{s}\in\Gamma(E^{s}),\ \ \ \nabla_{Y}X_{u}\in\Gamma(E^{u}).\]

\item The restriction of $\nabla$ to each leaf of the foliations $\mathcal{W}^{s}$
and $\mathcal{W}^{u}$ of $N$ is flat (note that restriction of the
connection to the leaves of the stable and unstable foliations is
smooth so it makes sense to talk about its curvature). 
\item Parallel transport along curves on the weak stable and unstable manifolds
coincide with the holonomy transport determined by the stable and
unstable foliations. 
\end{enumerate}
\end{thm}

\begin{rem} \label{-indeo of gamma}Let us observe that since we
know that the restriction of $\nabla$ to each leaf of $\mathcal{W}^{s}$
and $\mathcal{W}^{u}$ is flat, it follows that $H_{x,y}^{\gamma}$
is independent of $\gamma$. Thus we will omit $\gamma$ from the
notation and simply write $H_{x,y}$.\end{rem}

\begin{lem}\label{The-holonomy-transport=00003D00003D00003Ddif}The
holonomy transport $H_{x,\phi_{t}x}$ is given by $d_{x}\phi_{t}|E^{u}(x)$,
that is,\[
H_{x,\phi_{t}x}=d_{x}\phi_{t}|E^{u}(x)\ \ \ \mbox{as maps }E^{u}(x)\rightarrow E^{u}(\phi_{t}x).\]
 \end{lem}

\begin{proof}Fix $x\in N$, and let $\Gamma(t):=\phi_{t}x$. Fix
$\xi\in E^{u}(x)$, and let $ $$V(t)$ denote the vector field along
$\Gamma(t)$ defined by $V(t)=d_{x}\phi_{t}(\xi)$. It suffices to
show that $V$ is parallel: $\nabla_{\dot{\Gamma}}V\equiv0$.

Note that $\dot{\Gamma}(t)=F(\phi_{t}x)$. Since $\nabla_{F}=L_{F}$
we have \[
\nabla_{F(\phi_{t}x)}V(t)=\frac{d}{dt}\Bigl|_{t=0}d_{\phi_{t}x}\phi_{-t}(V(\phi_{t}x))=\frac{d}{dt}\Bigl|_{t=0}\xi=0.\]
 and the lemma follows.\end{proof}

\section{\label{sec:Constructing-the-Invariant}Constructing the Invariant
Subbundles}

Throughout this section let $\left(\Sigma,\omega\right)$ denote a
stable Anosov Hamiltonian structure of dimension $2n-1$ where $n\geq2$,
$\lambda$ a stabilizing $1$-form for $(\Sigma,\omega)$, $F$ the
Reeb vector field of $\lambda$, and $\phi_{t}$ the flow of $F$.
If $n$ is odd, suppose that the weak (un)stable bundles are of class
$C^{1}$. If $n$ is even, assume that the Anosov Hamiltonian structure
is $1/2$-pinched. In either case, the Kanai connection $\nabla$
is defined. The goal in this section is to construct a subbundle of
$E^{u}$ that is invariant under both $\phi_{t}$ and the holonomy
transport along the leaves of the weak stable foliation $\mathcal{W}^{+}$.
It is the existence of this subbundle that we will then exploit in
Section \ref{sec:Proof-of-Theorem A} in order to prove Theorem A
from the introduction. The main ideas for these constructions come
from \cite{Hamenstaedt1995,Feres1991,FeresKatok1990,Kanai1993}.\\

In the even dimensional case we will need to know that $d\lambda$
is parallel with respect to the Kanai connection. This is the only
place in the paper where we will actually use the $1/2$-pinching
condition (as opposed to just $C^{1}$ weak (un)stable bundles). The
following lemma is due to Kanai (\cite[Lemma 3.2]{Kanai1993}). \begin{lem}
\label{lem:.dlambda is parallel}Suppose $\phi_{t}$ is a time change
of a $1/2$-pinched Anosov flow. Then $\nabla(d\lambda)=0$.\end{lem}
\begin{proof} Suppose $\tau$ is any invariant $(0,3)$-tensor annihilated
by $F$, i.e. $i_{F}\tau=0$. We claim that $\tau$ must vanish. Note
that if $\psi_{t}$ is any time change of $\phi_{t}$, then $\psi_{t}$
also leaves $\tau$ invariant since $F$ annihilates $\tau$, so in
the proof below without loss of generality, we may assume that $\phi_{t}$
itself is $1/2$-pinched.

To see that $\tau$ vanishes consider for example a triple of vectors
$(\xi_{1},\xi_{2},\xi_{3})$ where $\xi_{1},\xi_{2}\in E^{s}(x)$
but $\xi_{3}\in E^{u}(x)$. Then there is a constant $C>0$ such that
\begin{align*}
|\tau_{x}(\xi_{1},\xi_{2},\xi_{3})| & =|\tau_{\phi_{t}x}(d_{x}\phi_{t}(\xi_{1}),d_{x}\phi_{t}(\xi_{2}),d_{x}\phi_{t}(\xi_{3}))|\\
 & \leq Ce^{(A-2a)t}|\xi_{1}||\xi_{2}||\xi_{3}|,\end{align*}
 By the $1/2$-pinching condition the last expression tends to zero
as $t\to\infty$ and therefore $\tau_{x}(\xi_{1},\xi_{2},\xi_{3})=0$.
The same will happen for other possible triples $(\xi_{1},\xi_{2},\xi_{3})$
when we let $t\to\pm\infty$.

Since $d\lambda$ and $\nabla$ are $\phi_{t}$-invariant, so is $\nabla(d\lambda)$.
Since $i_{F}d\lambda=0$, $\nabla(d\lambda)$ is also annihilated
by $F$ (to see that $\nabla_{F}(d\lambda)=0$ use that $d\lambda$
is $\phi_{t}$-invariant and that $\nabla_{F}=L_{F}$). Hence by the
previous argument applied to $\tau=\nabla(d\lambda)$ we conclude
that $\nabla(d\lambda)=0$ as desired.

\end{proof} \begin{lem} There exists a smooth $\phi_{t}$-invariant
bundle map $L:E\rightarrow E$ such that for $X,Y\in\Gamma(E)$, \[
d\lambda(X,Y)=\omega(LX,Y)=\omega(X,LY).\]
 Moreover $L$ preserves the decomposition of $E=E^{s}\oplus E^{u}$,
that is, $L=L^{s}+L^{u}$ where $L^{s}:E^{s}\rightarrow E^{s}$ and
$L^{u}:E^{u}\rightarrow E^{u}$.\end{lem} \begin{proof} Since $d\lambda$
is $\phi_{t}$-invariant and annihilated by $F$, there exists a $\phi_{t}$-invariant
smooth section $L$ of $E^{*}\otimes E$ such that the stated equation
holds. It remains to check that $L$ preserves the decomposition,
that is, $L$ commutes with $I$. But this is clear: if $X_{s},Y_{s}\in\Gamma(E^{s})$
then \[
d\lambda(X_{s},Y_{s})=X_{s}\lambda(Y_{s})-Y_{s}\lambda(X_{s})-\lambda\bigl(\left[X_{s},Y_{s}\right]\bigr);\]
 using integrability of $E^{s}$ and the fact that $\ker\,\lambda=E$,
we see that \[
0=d\lambda(X_{s},Y_{s})=\omega(LX_{s},Y_{s}),\]
 and hence $LX_{s}\in\Gamma(E^{s})$. The same argument applies with
sections of $E^{u}$. \end{proof} The construction of the invariant
subbundle will depend on the parity of $n$. We will begin with the
easier case, when $n$ is even.

\subsection*{The case of $n$ even}

Since $\dim\,\Sigma=2n-1$ and $n$ is even, we have $\dim\, E^{u}=n-1$
an odd number. Thus for any $x\in\Sigma$, the map $L_{x}^{u}:E^{u}(x)\rightarrow E^{u}(x)$
admits a real eigenvalue $\rho_{x}$. In fact, we have the following
result. \begin{lem} \label{lem:eigenvalue lemma}There exists $\rho\in\mathbb{R}$
such that $\rho_{x}=\rho$ for all $x\in\Sigma$.\end{lem}

\begin{rem}\label{standard anosov rem}In the proof of the lemma
we will make use of the fact that $\phi_{t}$ is a \emph{transitive}
flow. To see this, we first note that since $\phi_{t}$ preserves
a probability measure, the \emph{non-wandering set }$\Omega(\phi)$
of $\phi_{t}$ is necessarily equal to all of $\Sigma$ (see for instance
\cite{Anosov1967} or \cite[Chapter 18]{KatokHasselblatt1995}). It
is then a standard result that an Anosov flow whose non-wandering
set is the whole space is transitive (see for instance \cite[p.576]{KatokHasselblatt1995}).
We also remark that since $\phi_{t}$ is Anosov the set of periodic
points is dense in the non-wandering set $\Omega(\phi)$, and hence
dense in $\Sigma$; we will use this observation in the next subsection.\end{rem}

\begin{proof} For $k\geq0$, let $a_{k}(x)$ denote the coefficient
of $t^{k}$ in the characteristic polynomial $p_{x}(t)$ of $L_{x}^{u}$.
Then $a_{k}:\Sigma\rightarrow\mathbb{R}$ is continuous and $\phi_{t}$-invariant.
Since $\phi_{t}$ is transitive, $a_{k}$ is constant. Thus the characteristic
polynomial $p_{x}(t)$ of $L_{x}^{u}$ is independent of $x$, and
so each $L_{x}^{u}$ admits the same eigenvalues. \end{proof} Let
therefore $\rho_{0}\in\mathbb{R}$ be a common eigenvalue of the maps
$\left\{ L_{x}^{u}\right\} $. Let \begin{equation}
P_{\rho_{0}}(x):=\left\{ \xi\in E^{u}(x)\,:\, L_{x}^{u}\xi=\rho_{0}\xi\right\} \ne\{0\}.\label{eq:P ro 0}\end{equation}

\begin{prop} The map $x\mapsto P_{\rho_{0}}(x)$ defines a $C^{1}$
subbundle of $E^{u}$. Moreover for $x\in\Sigma$, the restriction
of $P_{\rho_{0}}$ to $W^{u}(x)$ is integrable. \end{prop} \begin{proof}
Since $d\lambda$ is $\nabla$-parallel by Lemma \ref{lem:.dlambda is parallel},
$P_{\rho_{0}}(x)$ is invariant under the parallel transport of $\nabla$.
More precisely, given a curve $\gamma$ from $x$ to $y$, if $\mathcal{P}_{x,y}^{\gamma}:T_{x}\Sigma\rightarrow T_{y}\Sigma$
denotes the parallel transport along $\gamma$ with respect to $\nabla$,
then \[
\mathcal{P}_{x,y}^{\gamma}P_{\rho_{0}}(x)\subseteq P_{\rho_{0}}(y),\]
 since for $\xi\in E^{u}(x)$, if $L_{x}^{u}\xi=\rho_{0}\xi$ then
\[
L_{y}^{u}\mathcal{P}_{x,y}^{\gamma}(\xi)=\rho_{0}\mathcal{P}_{x,y}^{u}(\xi).\]

Indeed, we have for any vector fields $X,Y,Z$ that:\begin{eqnarray*}
0 & = & \bigl(\nabla_{X}d\lambda\bigr)\left(Y,Z\right)\\
 & = & \nabla_{X}(d\lambda(Y,Z))-d\lambda\bigl(\nabla_{X}Y,Z\bigr)-d\lambda\bigl(Y,\nabla_{X}Z\bigr)\\
 & = & \nabla_{X}\bigl(\omega(LY,Z)\bigr)-\omega\bigl(L\bigl(\nabla_{X}Y\bigr),Z\bigr)-\omega\bigl(LY,\nabla_{X}Z\bigr)\\
 & = & \nabla_{X}\bigl(\omega\left(LY,Z\right)\bigr)-\omega\bigl(\nabla_{X}(LY),Z\bigr)-\omega\bigl(LY,\nabla_{X}Z\bigr)+\omega\bigl(\bigl(\nabla_{X}L\bigr)Y,Z\bigr)\\
 & = & \nabla_{X}\omega\bigl(LY,Z\bigr)+\omega\bigl(\bigl(\nabla_{X}L\bigr)Y,Z\bigr)\\
 & = & \omega\bigl(\bigl(\nabla_{X}L\bigr)Y,Z\bigr),\end{eqnarray*}
where the last equality used the fact that $\nabla\omega=0$. Thus
$L$ is parallel, and hence $L^{u}$ and $L^{s}$ are also parallel.
Thus: \begin{eqnarray*}
L_{y}^{u}\mathcal{P}_{x,y}^{\gamma}(\xi) & = & \mathcal{P}_{x,y}^{\gamma}\bigl(L_{x}^{u}\xi\bigr)\\
 & = & \mathcal{P}_{x,y}^{\gamma}(\rho_{0}\xi)\\
 & = & \rho_{0}\mathcal{P}_{x,y}^{\gamma}(\xi).\end{eqnarray*}
 It remains to check that the restriction of $P_{\rho_{0}}$ to $W^{u}(x)$
is integrable. Note that if $X$ and $Y$ are parallel sections of
$E^{u}$ over $W^{u}(x)$ then since $\nabla$ has no torsion over
$E^{u}$, we have \[
0=\nabla_{X}Y-\nabla_{Y}X=\left[X,Y\right].\]

\end{proof}

\subsection*{The case of $n$ odd}

We now want to construct a $\phi_{t}$-invariant subbundle $P_{\rho_{0}}\subseteq E^{u}$
that is invariant under holonomy transport along the leaves of $\mathcal{W}^{+}$
for the case when $n$ is an odd integer. As before we construct the
maps $L_{x}^{u}:E^{u}(x)\rightarrow E^{u}(x)$; however in this case
since $\dim\, E^{u}(x)=n-1$ is even, it is no longer necessarily
the case that $L_{x}^{u}$ admits a real eigenvalue, and so our previous
construction will not work.

\begin{rem}\label{Referee remark}Recall we are only assuming that
$\phi_{t}$ is a time change of a $1/2$-pinched Anosov flow in the
even dimensional case. Thus in the odd dimensional case Lemma \ref{lem:.dlambda is parallel}
is not available to us. If however we did assume that $\nabla(d\lambda)=0$
then we could dramatically simplify the treatment of the odd-dimensional
case both in this section and in Section \ref{sec:Proof-of-Theorem A}.
Indeed, whilst in the odd-dimensional case the characteristic polynomial
$p(t)$ of $L^{u}$ no longer necessarily admits a real eigenvalue,
it is reducible over $\mathbb{R}$ to a product of quadratic factors,
say $p(t)=q_{1}(t)\dots q_{k}(t)$. We may then define sub-bundles
$P_{i}:=\ker\, q_{i}(L^{u})$ of $E^{u}$, which have constant non-zero
dimension. Since $\nabla(d\lambda)=0$ these are parallel and hence
integrable. We thank the referee for this observation.\end{rem}

Let $E_{\mathbb{C}}^{u}:=E^{u}\otimes\mathbb{C}$ denote the complexification
of $E^{u}$, and let $\mathbb{L}:E_{\mathbb{C}}^{u}\rightarrow E_{\mathbb{C}}^{u}$
denote the complex linear extension of $L^{u}$. Given $\rho\in\mathbb{C}$,
we set \[
Q_{\rho}(x):=\left\{ \xi\in E^{u}(x)\,:\,\xi=\mbox{Re}\,\zeta\mbox{ for some }\zeta\in E_{\mathbb{C}}^{u}\mbox{ with }\mathbb{L}_{x}\zeta=\rho\zeta\right\} .\]
 By the same arguments as those used in the proof of Lemma \ref{lem:eigenvalue lemma},
there exists an open $\phi_{t}$-invariant set $O\subseteq\Sigma$
with the property that for any fixed $\rho\in\mathbb{C}$, the dimension
of the subspaces $Q_{\rho}(x)$ is constant for $x\in O$. Thus $Q_{\rho}|O$
defines a continuous $\phi_{t}$-invariant subbundle of $E^{u}|O$
for all $\rho\in\mathbb{C}$.

Now define a new bundle $P_{\rho}$ over all of $\Sigma$ by setting
\begin{equation}
P_{\rho}(x):=\bigcap\left\{ H_{y,x}Q_{\rho}(y)\,:\, y\in O\cap W^{s}(x)\right\} .\label{eq:P ro}\end{equation}
 We call $P_{\rho}$ the \emph{holonomy intersection} of $Q_{\rho}$
(see \cite[p685]{Hamenstaedt1995}, as well as \cite[Section 8]{Feres1991}).
Note that $P_{\rho}$ is clearly holonomy invariant. The next result
is Lemma $2.7$ in \cite{Hamenstaedt1995}. \begin{thm} \label{thm:hamen2-1}For
all $\rho\in\mathbb{C}$, $P_{\rho}$ is a continuous $\phi_{t}$-invariant
subbundle of $E^{u}$ over $\Sigma$. \end{thm} \begin{proof} Let
\[
K:=\min\left\{ \dim\, P_{\rho}(x)\,:\, x\in\Sigma\right\} .\]
 Since the dimension of $K$ is constant along the leaves of $\mathcal{W}^{s}$,
there exists $x_{0}\in O$ with $\dim P_{\rho}(x_{0})=K$. We can
choose points $x_{1},\dots,x_{\ell}\in W^{s}(x_{0})\cap O$ such that
\[
P_{\rho}(x_{0})=\bigcap_{i=1}^{\ell}H_{x_{i},x_{0}}Q_{\rho}(x_{i}).\]
 Since $Q_{\rho}$ and the holonomy maps $H_{x_{i},x_{0}}$ are continuous
there exists an open neighborhood $O_{0}\subseteq O\cap W^{-}(x_{0})$
of $x_{0}$ such that the following holds. Given $y_{0}\in O_{0}$
there exist open neighborhoods $U_{i}$ of $x_{i}$ such that the
connected component of $x_{i}$ in the set $U_{i}\cap W^{s}(y_{0})\cap W^{u}(x_{i})$
consists of a single point $y_{i}$, and moreover that \[
\dim\,\bigcap_{i=1}^{\ell}H_{y_{i},y_{0}}Q_{\rho}(y_{0})\leq K.\]
 By minimality of $K$ this forces\[
P_{\rho}(y_{0})=\bigcap_{i=1}^{\ell}H_{y_{i},y_{0}}Q_{\rho}(y_{0});\]
 in particular this proves we can find a neighborhood $V$ of $x_{0}$
in $\Sigma$ such that the restriction of $P_{\rho}$ to $V$ is a
continuous subbundle of $E^{u}|V$. But then since \[
\Sigma=\bigcup_{y\in V}W^{s}(y),\]
 it follows that $P_{\rho}$ is a continuous subbundle of $E^{u}$
over all of $\Sigma$. \end{proof} Unfortunately, it is not necessarily
the case that $P_{\rho}$ is of positive dimension. To ensure this,
we will need to choose $\rho\in\mathbb{C}$ carefully. Here is the
general idea; the precise construction (due to Hamenst\"adt \cite[Section 2]{Hamenstaedt1995})
is somewhat technical. We choose a periodic point $q$ of $\phi_{t}$,
with period $T>0$, say. Then $d_{q}\phi_{T}:T_{q}\Sigma\rightarrow T_{q}\Sigma$
induces a map $A_{q}^{u}:E^{u}(q)\rightarrow E^{u}(q)$. Let $\mathbb{A}_{q}$
denote the complex linear extension of $A_{q}^{u}$ to a map $E_{\mathbb{C}}^{u}(q)\rightarrow E_{\mathbb{C}}^{u}(q)$,
and let $\sigma\in\mathbb{R},\sigma>1$ denote the minimal absolute
value of an eigenvalue of $\mathbb{A}_{q}$. We are interested in
the subspace $S(q)\subseteq E^{u}(q)$ consisting of the subspace
spanned by the union of the eigenspaces of $\mathbb{A}_{q}$ corresponding
to eigenvalues of absolute value $\sigma$.

We then use holonomy transport to carry $S(q)$ to subspaces $S(x)\subseteq E^{u}(x)$
for $x\in W^{+}(q)$, thus creating a distribution $S$ over $E^{u}|W^{+}(q)$.\\

What is the point of this construction? Suppose now $q\in O$ (where
$O$ is the open set defined earlier on which $Q_{\rho}$ is a continuous
subbundle of $E^{u}|O$ for all $\rho\in\mathbb{C}$ - such $q$ always
exist since the set of periodic points of $\phi_{t}$ is dense in
$\Sigma$, see Remark \ref{standard anosov rem}). Then one shows
that $L_{q}^{u}:E^{u}(q)\rightarrow E^{u}(q)$ preserves the subbundle
$S(q)$, and thus $\mathbb{L}_{q}:E_{\mathbb{C}}^{u}(q)\rightarrow E_{\mathbb{C}}^{u}(q)$
preserves the complexification $S_{\mathbb{C}}(q)\subseteq E_{\mathbb{C}}^{u}(q)$.
Moreover if $\rho_{0}$ is an eigenvalue of $\mathbb{L}_{q}|S_{\mathbb{C}}(q)$,
we will show that the subspace \[
Q_{\rho_{0}}(q)\cap S(q)\subseteq E^{u}(q)\]
 (which is necessarily of positive dimension) is contained in $P_{\rho_{0}}(q)$;
in other words, for this choice of $\rho_{0}$, $P_{\rho_{0}}$ is
of positive dimension, and this gives us our desired subbundle of
$E^{u}$ for the case where $n$ is odd.\\

We will now begin with the details of the construction.\\

Since $q$ is periodic of period $T$, $d_{q}\phi_{T}:T_{q}\Sigma\rightarrow T_{q}\Sigma$
defines a hyperbolic linear map $A_{q}:E(q)\rightarrow E(q)$ which
preserves $E^{s}(q)$ and $E^{u}(q)$, and so defines maps $A_{q}^{s}:E^{s}(q)\rightarrow E^{s}(q)$
and $A_{q}^{u}:E^{u}(q)\rightarrow E^{u}(q)$. Let $\mathbb{A}_{q}:E_{\mathbb{C}}^{u}(q)\rightarrow E_{\mathbb{C}}^{u}(q)$
denote the linear map induced by $A_{q}^{u}$. Let $\sigma\in\mathbb{R}$
denote the smallest absolute value of an eigenvalue of $\mathbb{A}_{q}$
(note $\sigma>1$ as $d\phi_{t}|E^{u}$ is expanding).

Now let us set \[
S(q):=\mbox{span}\left\{ \xi\in E^{u}(q)\,:\,\exists\,\zeta\in E_{\mathbb{C}}^{u}(q),\rho\in\mathbb{C},\left|\rho\right|=\sigma,\mbox{ with }\xi=\mbox{Re}\,\zeta,\ \mathbb{A}_{q}\zeta=\rho\zeta\right\} .\]

Then for $x\in W^{+}(q)$ we define \[
S(x):=H_{q,x}[S(q)].\]
 Then $S$ is a $C^{1}$-subbundle of $E^{u}|W^{+}(q)$ and moreover
using \eqref{eq:diff hol eq} and Lemma \ref{The-holonomy-transport=00003D00003D00003Ddif}
we see $d_{x}\phi_{t}[S(x)]\subseteq S(\phi_{t}x)$ for all $x\in W^{+}(q)$;
in particular $A_{q}^{u}$ maps $S(x)$ to itself.\\
 \\
 We now make a little digression into some elementary linear algebra.
Given an invertible complex linear endomorphism $A:\mathbb{C}^{m}\rightarrow\mathbb{C}^{m}$,
decompose $\mathbb{C}^{m}=\bigoplus_{j=1}^{k}V^{\rho_{j}}$ into the
root spaces of $A$, and let $\left\{ \xi_{1}^{\rho_{1}},\dots,\xi_{d(\rho_{1})}^{\rho_{1}},\xi_{1}^{\rho_{2}},\dots,\xi_{d(\rho_{2})}^{\rho_{2}},\dots,\xi_{1}^{\rho_{k}},\dots,\xi_{d(\rho_{k})}^{\rho_{k}}\right\} $
denote a basis of $\mathbb{C}^{m}$ such that $A$ is in Jordan normal
form with respect to this basis. That is, we have $d(\rho_{j})=\dim\, V^{\rho_{j}}$,
$A\xi_{1}^{\rho_{j}}=\rho_{j}\xi_{1}^{\rho_{j}}$ for $j=1,\dots,k$
and $A\xi_{i}^{\rho_{j}}=\xi_{1}^{\rho_{j}}+\xi_{2}^{\rho_{j}}+\dots+\xi_{i-1}^{\rho_{j}}+\rho_{j}\xi_{i}^{\rho_{j}}$
for $i>1$ and $j=1,\dots,k$. Now set \[
\eta_{i}^{\rho_{j}}:=\frac{\rho_{j}}{\left|\rho_{j}\right|}\xi_{i}^{\rho_{j}},\]
 and define a Hermitian inner product $\left\langle \cdot,\cdot\right\rangle $
on $\mathbb{C}^{m}$ by declaring $\{\eta_{i}^{\rho_{j}}\}$ to be
a unitary basis.

Let $\sigma:=\min\{\left|\rho_{j}\right|\,:j=1\dots,k\}$, and set
\[
S:=\mbox{span}_{\mathbb{C}}\left\{ \xi_{1}^{\rho_{j}}\,:\,\left|\rho^{j}\right|=\sigma\right\} .\]
 Then for any $\xi,\eta\in S$ we have \[
\left\langle A\xi,A\eta\right\rangle =\sigma^{2}\left\langle \xi,\eta\right\rangle ,\]
 The following result is essentially due to Hamenst\"adt.

\begin{lem}Let $\xi,\eta\in\mathbb{C}^{n}$, with $\eta\in S$, $\eta\ne0$
and $\xi\notin S$. Then \begin{equation}
\lim_{k\rightarrow\infty}\frac{\left|A^{k}\xi\right|}{\left|A^{k}\eta\right|}=\infty.\label{eq:required limit}\end{equation}
 \end{lem}

\begin{proof}The crux of the proof is the following formula, whose
proof can be found in \cite[Corollary 2.3]{Hamenstaedt1995}. Suppose
$\zeta\in\mathbb{C}^{n}$ is a root vector for $A$ with eigenvalue
$\rho\ne0$, that is, there exists $j\in\mathbb{N}$ such that\[
(A-\rho\mbox{\mbox{Id}})^{j}\zeta=0\ \ \ \mbox{but}\ \ \ (A-\rho\mbox{\mbox{Id}})^{j-1}\zeta\ne0.\]
 Then \[
\lim_{k\rightarrow\infty}\frac{A^{k}\zeta}{\rho^{k}k^{j-1}}=\frac{1}{\rho^{j-1}(j-1)!}(A-\sigma\mbox{Id})^{j-1}\zeta.\]
 In particular, if either: 
\begin{enumerate}
\item $\rho,\rho'$ are eigenvalues of $A$ with $\left|\rho\right|<\left|\rho'\right|$
and $\zeta,\zeta'\in\mathbb{C}^{N}$ are root vectors for $A$ with
eigenvalues $\rho,\rho'$ or 
\item $\zeta$ and $\zeta'$ are both root vectors for $A$ with eigenvalue
$\rho$, such that there exists $j\in\mathbb{N}$ with \[
(A-\rho\mbox{Id})^{j}\zeta=0\ \ \ \mbox{but}\ \ \ (A-\rho\mbox{Id})^{j}\zeta'\ne0,\]

\end{enumerate}
then \[
\lim_{k\rightarrow\infty}\frac{\left|A^{k}\zeta'\right|}{\left|A^{k}\zeta\right|}=\infty.\]
 This implies the lemma.\end{proof}

We return now to the problem at hand and prove the following result
(\cite[Lemma 2.5]{Hamenstaedt1995}). \begin{prop} \label{lu preserves}The
map $L^{u}$ preserves the bundle $S$ over $E^{u}|W^{+}(q)$. \end{prop}

Before getting started on the proof, we introduce an auxilliary inner
product that will be helpful for this result and its sequel. Let us
fix an inner product $\left\langle \cdot,\cdot\right\rangle _{q}$
on $E^{u}(q)$ with the property that if $\xi,\eta\in S(q)$ then
$\left\langle A_{q}^{u}\xi,A_{q}^{u}\eta\right\rangle _{q}=\sigma^{2}\left\langle \xi,\eta\right\rangle _{q}$
(such an inner product exists by the discussion above). Then extend
$\left\langle \cdot,\cdot\right\rangle _{q}$ to an inner product
$\left\langle \cdot,\cdot\right\rangle _{x}$ for all $x\in W^{s}(q)$
by\[
\left\langle \xi,\eta\right\rangle _{x}:=\left\langle H_{x,q}(\xi),H_{x,q}(\eta)\right\rangle _{q}\ \ \ \mbox{for }x\in W^{s}(q),\;\xi,\eta\in E^{u}(x).\]
 Let $\left|\cdot\right|_{x}$ denote the norm induced by $\left\langle \cdot,\cdot\right\rangle _{x}$
for $x\in W^{s}(q)$.

\begin{proof} \emph{(of Proposition \ref{lu preserves})}

First note that by \eqref{eq:diff hol eq} for any $x\in W^{s}(q)$
and $k\in\mathbb{Z}$, \begin{equation}
H_{\phi_{kT}x,q}\circ d_{x}\phi_{kT}=\left(A_{q}^{s}\right)^{k}\circ H_{x,q}\mbox{ as maps }E^{u}(x)\rightarrow E^{u}(q).\label{eq:important H+}\end{equation}
 Thus if $\xi\in E^{u}(x)$, $\xi\notin S(x)$ and $\eta\in S(x),\eta\ne0$
we have \begin{eqnarray*}
\frac{\left|d_{x}\phi_{kT}(\xi)\right|_{\phi_{kT}x}}{\left|d_{x}\phi_{kT}(\eta)\right|_{\phi_{kT}x}} & = & \frac{\left|H_{\phi_{kT}x,q}\bigl(d_{x}\phi_{kT}(\xi)\bigr)\right|_{q}}{\left|H_{\phi_{kT}x,q}\bigl(d_{x}\phi_{kT}(\eta)\bigr)\right|_{q}}\\
 & = & \frac{\left|\left(A_{q}^{u}\right)^{k}H_{x,q}(\xi)\right|_{q}}{\left|\left(A_{q}^{u}\right)^{k}H_{x,q}(\eta)\right|_{q}},\end{eqnarray*}
 and then since $H_{x,q}(\eta)\in S(q)$ but $H_{x,q}(\xi)\notin S(q)$,
the previous lemma tells us that\[
\lim_{k\rightarrow\infty}\frac{\left|d_{x}\phi_{kT}(\xi)\right|_{\phi_{kT}x}}{\left|d_{x}\phi_{kT}(\eta)\right|_{\phi_{kT}x}}=\infty.\]
 Now let $B$ denote the ball \[
B=\left\{ x\in W^{s}(q)\,:\,\mbox{dist}(q,x)\leq1\right\} .\]
 Using continuity and compactness of $B$, we see that the operator
norm of $L^{u}$ with respect to the norm $\left|\cdot\right|$ on
$E^{u}|B$ is uniformly bounded on $B$.

Suppose now for contradiction that there exists $x\in W^{s}(q)$ and
$\xi\in S(x)$ with $L_{x}^{u}\xi\notin S(x)$. Then since $L_{x}^{u}$
is $\phi_{t}$-invariant we have \[
L_{\phi_{kT}x}^{u}\left(d_{x}\phi_{kT}\left(\xi\right)\right)=d_{x}\phi_{kT}\left(L_{x}^{u}\xi\right)\]
 for all $k\geq0$, and thus \[
\lim_{k\rightarrow\infty}\frac{\left|L_{\phi_{kT}x}^{u}\bigl(d_{x}\phi_{kT}\left(\xi\right)\bigr)\Bigr)\right|_{\phi_{kT}x}}{\left|d_{x}\phi_{kT}\left(\xi\right)\right|_{\phi_{kT}x}}=\infty.\]
 In other words, the operator norms with respect to $\left|\cdot\right|$
of the endomorphisms $L_{\phi_{kT}x}^{u}$ of $E^{u}(\phi_{kT}x)$
tends to infinity as $k\rightarrow\infty$, contradicting the fact
that $\phi_{kT}x\in B$ for $k$ large enough.

Thus $L^{u}$ preserves $S$ over $E^{u}|W^{s}(q)$. To complete the
proof we need to show $L^{u}$ preserves $S$ over all of $E^{u}|W^{+}(q)$.
This however is clear, since $L^{u}$ is $\phi_{t}$-invariant.\end{proof}
\begin{prop} Given $x\in W^{+}(q)$ we have \[
L_{x}^{u}=H_{q,x}\circ L_{q}^{u}\circ H_{x,q}\mbox{ as maps }S(x)\rightarrow S(x).\]
 In other words, when restricted to the subbundle $S$ over $W^{+}(q)$,
the map $L^{u}$ `commutes with holonomy'. \label{prop:comm}\end{prop}
\begin{proof} As before, since $L^{u}$ commutes with $d\phi_{t}$
it suffices to verify the assertion for $x\in W^{s}(q)$. Thus given
$x\in W^{s}(q)$, define $C_{x}\in\mbox{End}\, S$ by \[
C_{x}=L_{x}^{u}-H_{q,x}\circ L_{q}^{u}\circ H_{x,q}.\]
 To complete the proof we show that $C_{x}=0$ for all $x\in W^{s}(q)$.
To do this we shall show that the function $\beta$ defined on $W^{s}(q)$
by \[
\beta(x)=\left\Vert C_{x}\right\Vert _{x}\]
 (where here $\left\Vert \cdot\right\Vert _{x}$ denotes the operator
norm on $E^{u}(x)$ with respect to $\left\langle \cdot,\cdot\right\rangle _{x}$)
is invariant under $\phi_{T}$. Since $\beta$ is continuous and $\beta(q)=0$
it then follows $\beta$ is identically zero, and hence $C_{x}=0$
for all $x\in W^{s}(q)$.

Since $L^{u}$ is $\phi_{t}$-invariant we have \[
C_{\phi_{T}x}\bigl(d_{x}\phi_{T}(\xi)\bigr)=d_{x}\phi_{T}\bigl(C_{x}(\xi)\bigr)\]
 for all $x\in W^{s}(q)$ and $\xi\in S(x)$. Since (using \eqref{eq:important H+})

\begin{eqnarray*}
\left|d_{x}\phi_{T}(\xi)\right|_{x} & = & \left|H_{x,q}\bigl(d_{x}\phi_{T}(\xi)\bigr)\right|_{q}\\
 & = & \left|A_{q}^{u}\bigl(H_{x,q}^{+}(\xi)\bigr)\right|_{q}\\
 & = & \sigma\left|H_{x,q}^{+}(\xi)\right|_{q}\\
 & = & \sigma\left|\xi\right|_{x},\end{eqnarray*}
 for all $\xi\in S(x)$ and $x\in W^{s}(q)$, for $k\geq0$ we have\begin{eqnarray*}
\frac{\left|C_{\phi_{kT}x}\circ d_{x}\phi_{kT}(\xi)\right|_{\phi_{kT}x}}{\left|d_{x}\phi_{kT}(\xi)\right|_{\phi_{kT}x}} & = & \frac{\left|d_{x}\phi_{kT}\circ C_{x}(\xi)\right|_{\phi_{kT}x}}{\sigma^{k}\left|\xi\right|_{x}}\\
 & = & \frac{\left|C_{x}(\xi)\right|_{x}}{\left|\xi\right|_{x}};\end{eqnarray*}
 which proves the assertions about the map $\beta$ stated above.
\end{proof} The point of all this work is the following result. \begin{cor}
\label{thm:hamen2}Let $\rho_{0}\in\mathbb{C}$ be an eigenvalue of
$\mathbb{L}_{q}|S_{\mathbb{C}}(q)$. Then the dimension of $P_{\rho_{0}}$
is strictly positive.\end{cor} \begin{proof} From the previous proposition
it follows that the non-trivial subspace $Q_{\rho_{0}}(q)\cap S(q)$
is contained in $P_{\rho_{0}}(q)$, and hence in the notation of Theorem
\ref{thm:hamen2-1}, the integer $K=\dim\, P_{\rho_{0}}$ is strictly
positive. \end{proof}

We conclude this section with the following construction, again due
to Hamenst\"adt \cite[p686]{Hamenstaedt1995}. If $V$ is a real vector
space and $V_{i}\subseteq V$ are even dimensional subspaces admitting
almost complex structures $\mathbb{J}_{i}$, let $\overline{\bigcap}_{i}(V_{i},\mathbb{J}_{i})$
denote the largest subspace $W\subseteq\bigcap_{i}V_{i}$ which is
invariant under the $\mathbb{J}_{i}$ and such that $\mathbb{J}_{i}|W=\mathbb{J}_{j}|W$
for all $i,j$. We call $\overline{\bigcap}_{i}(V_{i},\mathbb{J}_{i})$
the \emph{complex intersection} of the $(V_{i},\mathbb{J}_{i})$.

Now suppose the $\rho_{0}$ we found above happens to lie in $\mathbb{C}\backslash\mathbb{R}$.
Then if we consider the operator \[
\mathbb{J}:=\frac{1}{\mbox{Im}\,\rho_{0}}\left(L^{u}-(\mbox{Re}\,\rho_{0})\mbox{Id}\right),\]
 it is easy to see that $\mathbb{J}$ defines an almost complex structure
on $Q_{\rho_{0}}$.

If follows that if we define a new bundle $P_{\rho_{0}}^{\mathbb{C}}$
over all of $\Sigma$ by setting \begin{equation}
P_{\rho_{0}}^{\mathbb{C}}(x):=\overline{\bigcap}\left\{ (H_{y,x}Q_{\rho_{0}}(y),\mathbb{J}_{y})\,:\, y\in W^{s}(x)\right\} ,\label{P ro complex}\end{equation}
 then the same proof as above shows that $P_{\rho_{0}}^{\mathbb{C}}$
is a continuous $\phi_{t}$-invariant subundle of $E^{u}$ over $\Sigma$
of positive dimension. We shall use this observation later.

\section{\label{sec:North-South-Dynamics}North-South Dynamics}

In this section we return to the situation described in the introduction.
The main goal of this section is to prove Theorem \ref{thm:north south dynamic}
below; we will use heavily the assumption that $M$ admits a metric
of negative curvature. As stated in the introduction however it should
be possible to prove Theorem \ref{thm:north south dynamic} without
using this assumption.

Throughout this section $(M,g)$ denotes a closed $n$-dimensional
Riemannian manifold with tangent bundle $\pi:TM\rightarrow M$. We
first begin with a quick summary of the geometry of the tangent bundle
that we will need throughout what follows.

\subsection*{The geometry of $TM$}

The \emph{vertical bundle} $V\subseteq TTM$ is given by \[
V(v)=\ker\left\{ d_{v}\pi:T_{v}TM\rightarrow T_{x}M\right\} ,\]
 where for convenience throughout this paragraph an arbitrary vector
$v\in TM$ is assumed to lie in $T_{x}M$. The Riemannian metric $g$
on $M$ determines a direct summand $H$ of the vertical bundle $V$,
called the \emph{horizontal bundle} together with isomorphisms\[
T_{v}TM\cong H(v)\oplus V(v)\cong T_{x}M\oplus T_{x}M.\]
 Let $\nabla$ denote the Levi-Civita connection on $(M,g)$. We make
this isomorphism explicit as follows: given $\xi\in T_{v}TM$ associate
a point $(\xi_{H},\xi_{V})\in T_{x}M\oplus T_{x}M$, where \[
\xi_{H}=d_{v}\pi(\xi)\]
 and \[
\xi_{V}=K_{v}(\xi).\]
 Here $K:TTM\rightarrow TM$ is the connection map of $\nabla$, defined
as follows: given $\xi\in T_{v}TM$, choose a curve $Z:(-\varepsilon,\varepsilon)\rightarrow TM$
such that $Z(0)=v$ and $\dot{Z}(0)=\xi$. Then \[
K_{v}(\xi)=\nabla_{t}Z(0),\]
 where $\nabla_{t}$ denotes the covariant derivative along the curve
$\pi\circ Z$.

$H(v)$ is thus defined to be the set of $\xi\in T_{v}TM$ such that
$\xi_{V}=0$, and similarly $V(v)$ is simply the set of $\xi\in T_{v}TM$
such that $\xi_{H}=0$. Clearly the map $\xi\mapsto\xi_{H}$ defines
an isomorphism $H(v)\rightarrow T_{x}M$ and similarly $\xi\mapsto\xi_{V}$
defines an isomorphism $V(v)\rightarrow T_{x}M$. In general we shall
slightly abuse notation and write $\xi=(\xi_{H},\xi_{V})$ to indicate
this identification.\\

It is easy to see that given $\xi,\eta\in TTM$ we have\[
\omega_{0}(\xi,\eta)=\left\langle \xi_{H},\eta_{V}\right\rangle -\left\langle \xi_{V},\eta_{H}\right\rangle ,\]
 where as before $\omega_{0}$ denotes the canonical symplectic form
on $TM$. We define the \emph{Sasaki metric} $g_{TM}$ on $TM$ by
setting \[
\left\langle \xi,\eta\right\rangle _{TM}:=\left\langle \xi_{H},\eta_{H}\right\rangle +\left\langle \xi_{V},\eta_{V}\right\rangle ,\]
 so that $\omega_{0}$ and $g_{TM}$ are compatible.

\subsection*{North-South dynamics}

As before, let $H:TM\rightarrow M$ denote the energy Hamiltonian
$(x,v)\mapsto\frac{1}{2}\left|v\right|^{2}$, and let $F_{1}$ denote
the symplectic gradient of $H$ with respect to $\omega_{1}$. Let
$\phi_{t}^{1}:TM\rightarrow TM$ denote the flow of $F_{1}$. Let
$\Sigma_{k}=H^{-1}(k)$ denote a closed energy level set, where $k$
is a regular value of $H$, and assume that $\phi_{t}^{1}|\Sigma_{k}$
is Anosov (we do not need to assume $C^{1}$ weak (un)stable bundles
at this point). Write the Anosov splitting of $T\Sigma_{k}$ as \[
T\Sigma_{k}=\mathbb{R}F_{1}\oplus E_{1}^{s}\oplus E_{1}^{u}.\]
 We first quote the following theorem from \cite{PaternainPaternain1994},
which will very useful in what follows. \begin{thm} \label{thm:pat&pat theorem}If
$\Sigma_{k}$ is an Anosov energy level then the weak (un)stable bundles
$E_{1}^{+}$ and $E_{1}^{-}$ are transverse to the vertical subbundle
$V$. \end{thm} Now let $\widetilde{M}$ denote the universal covering
of $M$, and let $\widetilde{\Sigma}_{k}$ denote the pullback of
$\Sigma_{k}$ to $T\widetilde{M}$. Then $\widetilde{\Sigma}_{k}$
is a smooth connected hypersurface of $T\widetilde{M}$ that intersects
each tangent space $T_{x}\widetilde{M}$ in a sphere containing the
origin in its interior (since the same is true of $\Sigma_{k}$).
Let $\widetilde{\sigma}$ denote the pullback of $\sigma$ to $\widetilde{M}$
and let $\widetilde{\omega}_{0}$ denote the natural symplectic form
on $T\widetilde{M}$. Let $\widetilde{\omega}_{1}:=\widetilde{\omega}_{0}+\widetilde{\pi}^{*}\widetilde{\sigma}$,
where $\widetilde{\pi}:T\widetilde{M}\rightarrow\widetilde{M}$ is
the footpoint map. We will let $\widetilde{F}_{1}$ denote the symplectic
gradient of the lifted Hamiltonian $(x,v)\mapsto\frac{1}{2}\left|v\right|^{2}$
with respect to $\widetilde{\omega}_{1}$. We will write $\widetilde{\phi}_{t}^{1}:T\widetilde{M}\rightarrow T\widetilde{M}$
for the flow of $\widetilde{F}_{1}$. By assumption $\widetilde{\phi}_{t}^{1}|\widetilde{\Sigma}_{k}$
is Anosov, and we will write the Anosov splitting as \[
T\widetilde{\Sigma}_{k}=\mathbb{R}\widetilde{F}_{1}\oplus\widetilde{E}_{1}^{s}\oplus\widetilde{E}_{1}^{u}.\]
 Similarly let us denote by $\widetilde{V}$ and $\widetilde{H}$
denote the vertical and horizontal subbundles of $TT\widetilde{M}$.\\

As before let $\mathcal{W}^{s},\mathcal{W}^{u},\mathcal{W}^{+}$ and
$\mathcal{W}^{-}$ denote the four foliations of $\Sigma_{k}$ defined
by the subbundles $E^{s},E^{u},E^{+}$ and $E^{-}$ respectively.
We can lift these to foliations $\widetilde{\mathcal{W}}^{s},\widetilde{\mathcal{W}}^{u},\widetilde{\mathcal{W}}^{+}$
and $\widetilde{\mathcal{W}}^{-}$ of $\widetilde{\Sigma}_{k}$. Let
$\mathcal{L}^{+}=\widetilde{\Sigma}_{k}/\widetilde{\mathcal{W}}^{+}$
and $\mathcal{L}^{-}=\widetilde{\Sigma}_{k}/\mathcal{\widetilde{\mathcal{W}}}^{-}$
denote the spaces of weak stable and unstable leaves respectively.
The fundamental group $\pi_{1}(M)$ (regarded as covering transformations
of $\widetilde{M}$) acts on $\widetilde{\Sigma}_{k}$ freely and
properly discontinuously by permuting the orbits of $\widetilde{\phi}_{t}^{1}$.
Since elements of $\pi_{1}(M)$ act by isometries, the action on $\widetilde{\Sigma}_{k}$
must send weak (un)stable leaves to weak (un)stable leaves, and thus
$\pi_{1}(M)$ induces an action on $\mathcal{L}^{+}$ and $\mathcal{L}^{-}$.\\

The aim of this section is to prove the following result. \begin{thm}
\label{thm:north south dynamic}Suppose $M$ admits a metric of negative
curvature. The fundamental group $\pi_{1}(M)$ acts on both $\mathcal{L}^{+}$
and $\mathcal{L}^{-}$ as a `North-South dynamic'. By this we mean
the following: for all $\varphi\in\pi_{1}(M)$, there exists two fixed
leaves $\widetilde{W}_{1}^{+},\widetilde{W}_{2}^{+}\in\mathcal{L}^{+}$
and two fixed leaves $\widetilde{W}_{1}^{-},\widetilde{W}_{2}^{-}\in\mathcal{L}^{-}$
such that for all $\widetilde{W}^{\pm}\in\mathcal{L}^{\pm}$ it holds
that \[
\lim_{n\rightarrow\infty}\varphi^{n}[\widetilde{W}^{\pm}]=\widetilde{W}_{1}^{\pm},\ \ \ \lim_{n\rightarrow\infty}\varphi^{-n}[\widetilde{W}^{\pm}]=\widetilde{W}_{2}^{\pm}.\]

\end{thm} Consider the fibration \[
\pi|\Sigma_{k}:\Sigma_{k}\rightarrow M\]
 of $\left(n-1\right)$-spheres. As a direct consequence of Theorem
\ref{thm:pat&pat theorem}, we see that the foliations $\mathcal{W}^{+}$
and $\mathcal{W}^{-}$ are transverse to the fibres of the fibration
$\Sigma_{k}\rightarrow M$. This implies that $\widetilde{\mathcal{W}}^{+}$
and $\widetilde{\mathcal{W}}^{-}$ are transverse to the fibration
$\widetilde{\pi}|\widetilde{\Sigma}_{k}:\widetilde{\Sigma}_{k}\rightarrow\widetilde{M}$.

Given $x\in\widetilde{M}$, the fibre $(\widetilde{\pi}|\widetilde{\Sigma}_{k})^{-1}(x)$$ $
is compact. Thus we can apply the following theorem of Ehresmann.
\begin{thm} Let $F\rightarrow E\overset{p}{\rightarrow}B$ be a fibre
bundle and $\mathcal{F}$ a foliation of $E$ transverse to the fibres.
Suppose $F$ is compact. Then for every leaf $L$ of $\mathcal{F}$,
$p|L:L\rightarrow B$ is a covering map. \end{thm} For a proof, see
\cite[p91]{CamachoNeto1985}. Now fix $v\in\widetilde{\Sigma}_{k}$.
Since $\widetilde{M}$ is simply connected, $\widetilde{\pi}|\widetilde{W}^{+}(v)$$ $
is a diffeomorphism, and thus $\widetilde{W}^{+}(v)$ simply connected.
Thus $\widetilde{W}^{+}(v)$ intersects each fibre of the fibration
$\widetilde{\Sigma}_{k}\rightarrow\widetilde{M}$ in precisely one
point, and thus each leaf $\widetilde{W}^{+}(v)$ is diffeomorphic
to $\widetilde{M}$. Thus $\widetilde{M}$ is diffeomorphic to $\mathbb{R}^{n}$,
and the space of stable leaves $\mathcal{L}^{+}$ can be identified
topologically with the $\left(n-1\right)$-sphere. Of course the same
applies to $\mathcal{L}^{-}$.\\

Now we recall the concept of the ideal boundary of $\widetilde{M}$.
For this, let $g_{0}'$ denote a metric of negative curvature on $M$
(whose existence we assume in this section), and lift $g_{0}'$ to
a metric $g_{0}$ on $\widetilde{M}$ of negative curvature. \begin{defn}
The \emph{ideal boundary} \textbf{\emph{$\widetilde{M}_{g_{0}}(\infty)$}}
of $(\widetilde{M},g_{0})$ is given by $\widetilde{M}_{g_{0}}(\infty):=\Lambda_{g_{0}}(\widetilde{M})/\sim$,
where $\Lambda_{g_{0}}(\widetilde{M})$ denotes the set of $g_{0}$-geodesics%
\footnote{For clarity we will use the letter $c$ to stand for $g_{0}$-geodesics
and $\gamma$ for the projection to $\widetilde{M}$ of flow lines
of $\widetilde{\phi}_{t}^{1}$.%
} $c:\mathbb{R}\rightarrow\widetilde{M}$ of $\widetilde{M}$, and
$c_{1}\sim c_{2}$ if and only if \[
\mbox{dist}_{\textrm{HD}}(c_{1}[\mathbb{R}^{+}],c_{2}[\mathbb{R}^{+}])<\infty;\]
 here $\mbox{dist}_{\textrm{HD}}$ denotes the \emph{Hausdorff distance}
defined by \[
\mbox{dist}_{\textrm{HD}}(U,V):=\inf\left\{ r\in\mathbb{R}\,:\, U\subseteq B(V,r),V\subseteq B(U,r)\right\} ,\]
 where $U,V\subseteq\widetilde{M}$ and $B(U,r):=\left\{ x\in\widetilde{M}\,:\,\mbox{dist}_{g_{0}}(x,U)\leq r\right\} $.
\end{defn} Given $x\in\widetilde{M}$ and $v\in T_{x}\widetilde{M}$,
let $c_{v}:\mathbb{R}\rightarrow\widetilde{M}$ denotes the unique
$g_{0}$-geodesic adapted to $v$, and let $c_{v}(\infty)\in\widetilde{M}_{g_{0}}(\infty)$
denote the corresponding element of $\widetilde{M}_{g_{0}}(\infty)$.
If $c_{v}^{-1}$ is the geodesic obtained by going along $c_{v}$
backwards, let $c_{v}(-\infty)$ denote the element of $\widetilde{M}_{g_{0}}(\infty)$
corresponding to $c_{v}^{-1}$.

Let $S^{g_{0}}\widetilde{M}$ denote the unit sphere bundle of $(\widetilde{M},g_{0})$.
Fix a point $x\in\widetilde{M}$, and consider the map $s_{x}:S_{x}^{g_{0}}\widetilde{M}\rightarrow\widetilde{M}_{g_{0}}^{g_{0}}(\infty)$
sending $v\mapsto c_{v}(\infty)$. Then $s_{x}$ is a bijection, and
we define a topology on $\widetilde{M}_{g_{0}}(\infty)$ so that that
$s_{x}$ becomes a homeomorphism; thus $\widetilde{M}_{g_{0}}(\infty)\cong S^{n-1}$.
This topology is independent of the choice of $x$, since $s_{y}\circ s_{x}^{-1}:S_{x}^{g_{0}}\widetilde{M}\rightarrow S_{y}^{g_{0}}\widetilde{M}$
is a homeomorphism.\\

The next thing we require is the concept of a quasi-geodesic. \begin{defn}
A curve $\gamma:[a,b]\rightarrow\widetilde{M}$ is an \emph{quasi-geodesic}
of $(\widetilde{M},g_{0})$ if there exist $P,Q\in\mathbb{R}^{+}$
such that \[
\frac{1}{P}\left|s-t\right|-Q\leq\mbox{dist}_{g_{0}}(\gamma(s),\gamma(t))\leq P\left|s-t\right|+Q\]
 for all $s,t\in[a,b]$. If we need to be explicit about the constants
$P,Q$, we call such a quasi-geodesic a \emph{$(P,Q)$-quasi-geodesic}.
\end{defn} We now quote two theorems which explain why this is relevant
to the situation in hand. The first is due to Peyerimhoff and Siburg
(\cite[Theorem 2.9]{PeyerimhoffSiburg2003}). \begin{thm} \label{thm:Pey and Sib}Suppose
$\phi_{t}^{1}|\Sigma_{k}$ is Anosov. Then there exists a constant
$P_{k}\in\mathbb{R}^{+}$ such that the projection to $\widetilde{M}$
of any orbit of $\widetilde{\phi}_{t}^{1}$ is a $(P_{k},0)$-quasi-geodesic.
\end{thm} \begin{rem} In fact, Theorem \ref{thm:Pey and Sib} is
stated in a somewhat different form in \cite{PeyerimhoffSiburg2003}:
there they assert that when $k$ is greater than a certain critical
value $c(g,\sigma)$ known as \emph{Ma\~n\'e's critical value} then the
projection to $\widetilde{M}$ of any minimizing orbit of $\widetilde{\phi}_{t}^{1}$
is a quasi-geodesic. See for instance \cite{ContrerasIturriagaPaternainPaternain1998,BurnsPaternain2002}
for the definition of $c(g,\sigma)$, where it is proved that when
$\phi_{t}^{1}$ is Anosov, it necessarily holds that $k>c(g,\sigma)$,
and that in this case, every orbit of $\widetilde{\phi}_{t}^{1}$
is minimizing. \end{rem} We can build the \emph{quasi-ideal boundary}
$\widetilde{M}_{g_{0}}^{*}(\infty)$ in much the same way using quasi-geodesics.
If $\gamma:\mathbb{R}\rightarrow\widetilde{M}$ is an $g_{0}$-quasi-geodesic,
we write $\gamma^{*}(\infty)$ to denote the corresponding element
of $\widetilde{M}_{g_{0}}^{*}(\infty)$. Note that any geodesic is
automatically a quasi-geodesic, and thus we have a natural map $\widetilde{M}_{g_{0}}(\infty)\hookrightarrow\widetilde{M}_{g_{0}}^{*}(\infty)$
carrying an equivalence class of geodesics to the corresponding equivalence
class of quasi-geodesics. This is the second theorem we quote here;
a proof may be found in \cite[Theorem 2.2]{Knieper2002}. \begin{thm}
\label{thm:Knei}The inclusion $\widetilde{M}_{g_{0}}(\infty)\hookrightarrow\widetilde{M}_{g_{0}}^{*}(\infty)$
is a bijection. \end{thm} We will use this to show the following
key result, whose proof is essentially that of Theorem $2.12$ in
\cite{Knieper2002}. The result however is originally due to Gromov
(see \cite{Gromov2000}), and also independently due to Ghys (\cite[Theorem 4.5]{Ghys1984}).
Let $\psi_{t}$ denote the geodesic flow of $(\widetilde{M},g_{0})$
and let $\psi_{t}'$ denote the geodesic flow of $(M,g_{0}')$. Let
$S^{g_{0}'}M$ denote the unit sphere bundle of $(M,g_{0}')$. 

\begin{thm} \label{thm:gromov}$\phi_{t}^{1}|\Sigma_{k}$ and $\psi'_{t}|S^{g_{0}'}M$
are topologically conjugate.\end{thm} \begin{proof} Given $v\in\widetilde{\Sigma}_{k}$,
let $\gamma_{v}:=\widetilde{\pi}\circ\widetilde{\phi}_{t}^{1}v$.
Then $\gamma_{v}$ determines an element $\gamma_{v}^{*}(\infty)\in\widetilde{M}_{g_{0}}^{*}(\infty)$
by Theorem \ref{thm:Pey and Sib}, and thus by Theorem \ref{thm:Knei}
a unique element $\xi_{v}\in\widetilde{M}_{g_{0}}(\infty)$. Let $\xi_{v}^{-1}\in\widetilde{M}_{g_{0}}(\infty)$
denote the element corresponding to $\gamma_{v}^{*}(-\infty)\in\widetilde{M}_{g_{0}}^{*}(\infty)$.

Suppose $\zeta,\xi\in\widetilde{M}_{g_{0}}(\infty)$. Then there exists
a unique $g_{0}$-geodesic $c$ such that $c(\infty)=\zeta$ and $c(-\infty)=\xi$.
Let $\mathbb{P}(c):\widetilde{M}\rightarrow\widetilde{M}$ denote
orthogonal projection onto $c$, and use this to define a map $\mathbb{P}(\zeta,\xi):\widetilde{M}\rightarrow S^{g_{0}}\widetilde{M}$
by \[
\mathbb{P}(\zeta,\xi)(x)=\dot{c}(t)\ \ \ \mbox{where\ \ \ }\mathbb{P}(c)(x)=c(t).\]
 Now define $G_{0}:\widetilde{\Sigma}_{k}\rightarrow S^{g_{0}}\widetilde{M}$
by setting \[
G_{0}(v):=\mathbb{P}(\xi_{v},\xi_{v}^{-1})(\widetilde{\pi}v),\]
 Then $G_{0}$ is continuous and surjective but in general not injective:
there may exist two points $v,v'$ on the same orbit of $\widetilde{\phi}_{t}^{1}$
that have the same orthogonal projection onto the $g_{0}$-geodesic
$c$ determined by $\xi_{v}=\xi_{v'}$ and $\xi_{v}^{-1}=\xi_{v'}^{-1}$.
In order to achieve injectivity we `average' $G_{0}$. For this look
at the map $\rho:\mathbb{R}\times\widetilde{\Sigma}_{k}\rightarrow S^{g_{0}}\widetilde{M}$
defined by \[
G_{0}(\widetilde{\phi}_{t}^{1}v)=\psi_{\rho(t,v)}(G_{0}(v)).\]
 Then $\rho$ satisfies the \emph{cocycle property}, that is, \[
\rho(t+t',v)=\rho(t,\widetilde{\phi}_{t'}^{1}v)+\rho(t',v),\]
 as is easily checked. Now choose $\tau\in\mathbb{R}^{+}$ such that
$\rho(\tau,v)>0$ for all $v\in\widetilde{\Sigma}_{k}$, and then
let $r(v)$ denote the average\[
r(v):=\frac{1}{\tau}\int_{0}^{\tau}\rho(t,v)dt.\]
 Next define $G_{\tau}:\widetilde{\Sigma}_{k}\rightarrow S^{g_{0}}\widetilde{M}$
by \[
G_{\tau}(v)=\psi_{r(v)}(G_{0}(v)).\]
 We claim that $G_{\tau}$ is injective. For this observe that if
\[
f(t):=r(\widetilde{\phi}_{t}^{1}v)+\rho(t,v)\]
 then $f$ is monotone increasing. Indeed, \begin{eqnarray*}
f'(t) & = & \frac{1}{\tau}\int_{0}^{\tau}\rho'(u+t,v)du\\
 & = & \frac{1}{\tau}(\rho(\tau+t,v)-\rho(t,v))\\
 & = & \frac{1}{\tau}\rho(\tau,\widetilde{\phi}_{t}^{1}v)>0.\end{eqnarray*}
 The claim then follows from the computation\begin{eqnarray*}
G_{\tau}(\widetilde{\phi}_{t}^{1}v) & = & \psi_{r(\psi_{t}v)}(G_{0}(\widetilde{\phi}_{t}^{1}v))\\
 & = & \psi_{r(\psi_{t}v)+\rho(t,v)}(G_{0}(v))\\
 & = & \psi_{f(t)}(G_{0}(v)).\end{eqnarray*}
 Finally, in order to deduce the stronger statement that $\phi_{t}^{1}|\Sigma_{k}$
and $\psi'_{t}|S^{g_{0}'}M$ are also topologically conjugate, one
simply notes that $G_{0}$ is obviously equivariant under the action
of $\pi_{1}(M)$, and hence so is $G_{\tau}$; thus $G_{\tau}$ descends
to $M$ to define an orbit equivalence $G'_{\tau}$ from $\phi_{t}^{1}|\Sigma_{k}$
and $\psi'_{t}|S^{g_{0}'}M$. \end{proof} It is now easy to prove
Theorem \ref{thm:north south dynamic}. Indeed, it is well known (see
for instance \cite[Theorem 3.8.13]{Klingenberg1995}) that Theorem
\ref{thm:north south dynamic} holds in the case of a geodesic flow
of a negatively curved manifold. Thus if $M$ admits a metric $g_{0}'$
of negative curvature, Theorem \ref{thm:gromov} gives us an orbit
equivalence between $\phi_{t}^{1}|\Sigma_{k}$ and the geodesic flow
$\psi_{t}'$ of $(M,g_{0}')$, and via this orbit equivalence we see
Theorem \ref{thm:north south dynamic} holds in our case too.

\section{\label{sec:Proof-of-Theorem A}Proof of Theorem A}

We will now prove Theorem A. Our proof of the theorem will depend
on the parity of $n$; moreover, a separate argument will be required
to deal with the cases $n=3$ and $n=7$. We will start with the case
where $n$ is even.

\subsection*{The case when $n$ is an even integer.}

In the even dimensional case, recall that we assume $\phi_{t}^{1}|\Sigma_{k}$
is $1/2$-pinched. Since $(\Sigma_{k},\omega_{1})$ is stable, there
exists an invariant subbundle $P_{\rho_{0}}$ of $E^{u}$ as constructed
in Section \ref{sec:Constructing-the-Invariant}; see \eqref{eq:P ro 0}.
The maximal integral submanifolds of $P_{\rho_{0}}$ define a foliation
$\mathcal{P}(v)$ of class $C^{1}$ on $W^{u}(v)$. Since $P$ is
invariant under parallel transport, it is also invariant under holonomy
transport and thus the foliations $\mathcal{P}(v)$ glue together
to give a foliation $\mathcal{P}$ of class $C^{1}$ on $\Sigma_{k}$
that is invariant under the parallel transport of $\nabla$ and thus
also the holonomy maps. Thus $\mathcal{P}$ can be lifted to $\widetilde{\Sigma}_{k}$
and then projected to a foliation $\mathcal{P}'$ of $\mathcal{L}^{+}$
of positive dimension.

Since $\mathcal{P}$ is invariant under $\phi_{t}$, $\mathcal{P}'$
is invariant under the action of $\pi_{1}(M)$. Here Theorem \ref{thm:north south dynamic}
of the previous section comes into play: $\pi_{1}(M)$ acts on $\mathcal{L}^{+}$
as a North-South dynamics. A theorem of Foulon \cite{Foulon1994}
states that there are no non-trivial $C^{0}$ foliations of the sphere
$S^{n-1}$ which are invariant under North-South dynamics. Since we
know that $\mathcal{P}'$ is of positive dimension, we must have $\mathcal{P}'=\mathcal{L}^{+}$,
and hence the subbundle $P_{\rho_{0}}$ is equal to $E^{u}$. From
this it is easy to deduce that $d\lambda=\rho_{0}\omega_{1}$, where
$\rho_{0}\in\mathbb{R}$ defines $P_{\rho_{0}}$; see \eqref{eq:P ro 0}.

To complete the proof in this case it remains to rule out the possibility
that $\rho_{0}=0$. Suppose for contradiction that this is the case.
Then the $1$-form $\lambda$ is closed. Let $\mu$ denote a Borel
probability measure on $\Sigma_{k}$. Recall that $\mu$ determines
a \emph{$1$-current} $l_{\mu}$ by \[
l_{\mu}\left(\beta\right)=\int_{\Sigma_{k}}\beta\left(F_{1}\right)d\mu\,:\,\beta\in\Omega^{1}(\Sigma_{k}).\]
 We say that $\mu$ is \emph{exact as a current} if $l_{\mu}\left(\beta\right)=0$
whenever $\beta$ is closed. Now let $\mu_{L}$ denote the Liouville
measure of $\Sigma_{k}$ (defined precisely below). Then by Lemma
\ref{lem:exact as a current} below, $\mu_{L}$ is exact as a current.
But \[
l_{\mu_{L}}\left(\lambda\right)=\int_{\Sigma_{k}}\lambda\left(F_{1}\right)d\mu_{L}\ne0;\]
 contradiction. Thus $\rho_{0}\ne0$; this completes the proof in
the even dimensional case.\\

For completeness let us give a precise definition of the Liouville
measure $\mu_{L}$ and prove that it is indeed exact as a current.
Whilst in general there may exist many invariant volume forms on an
energy level $\Sigma_{k}$, and thus many invariant probability measures,
in the special case where the energy level $\Sigma_{k}$ is Anosov,
the Liouville measure is the unique smooth invariant probability measure.
It can be defined as follows. Let $X\in TTM|\Sigma_{k}$ denote a
vector field such that $\omega_{1}(X,F_{1})=1$ (such a vector field
always exists since $\Sigma_{k}$ is a regular energy level). Observe
that \[
i_{F_{1}}i_{X}\omega_{1}^{n}|\Sigma_{k}=\omega_{1}^{n-1},\]
 and hence if $\Theta:=i_{X}\omega_{1}^{n}$ then $\Theta$ is a volume
form on $\Sigma_{k}$. Now observe that $\omega_{1}^{n-1}$ is exact
for $n\geq3$. Indeed \[
\omega_{1}^{n-1}=(\omega_{0}+\pi^{*}\sigma)^{n-1}=(\omega_{0})^{n-1}+(n-1)\pi^{*}\sigma\wedge(\omega_{0})^{n-2}.\]
 On the right-hand side the first term is exact. For $n\geq3$ the
second term is exact as well and the claim follows (when $n=2$, $\omega_{1}$
is exact if $M$ is not the 2-torus, but we do not need this here).
Then $i_{F_{1}}\Theta=d\tau$ and $\phi_{t}^{1}$ preserves the volume
form $\Theta$. Let $\mu_{L}$ denote the smooth invariant probability
measure induced by $\Theta$; $\mu_{L}$ is called the Liouville measure.
\begin{lem} \label{lem:exact as a current}The Liouville measure
$\mu_{L}$ of $\Sigma_{k}$ is exact as a current. \end{lem} \begin{proof}
Let $\beta\in\Omega^{1}(\Sigma_{k},\mathbb{R})$ denote any closed
$1$-form and let $A$ be the integral of $\Theta$. Then\begin{eqnarray*}
A\, l_{\mu_{L}}(\beta) & = & A\,\int_{\Sigma_{k}}\beta(F_{1})d\mu_{L}\\
 & = & \int_{\Sigma_{k}}\beta(F_{1})\Theta\\
 & = & \int_{\Sigma_{k}}i_{F_{1}}\Theta\wedge\beta\\
 & = & \int_{\Sigma_{k}}d\tau\wedge\beta\\
 & = & \int_{\Sigma_{k}}d(\tau\wedge\beta),\end{eqnarray*}
 and this last integral is zero by Stokes' theorem. \end{proof}

\subsection*{The case when $n$ is an odd integer.}

We now proceed to the second case, where $n$ is odd. We no longer
need to assume that $\phi_{t}^{1}|\Sigma_{k}$ is $1/2$-pinched,
only that the weak (un)stable bundles are of class $C^{1}$. The next
result is from \cite[Lemma 2]{Feres1991}. \begin{lem} \label{lem:useful lemma from Feres}Let
$K\subseteq\widetilde{E}_{1}^{u}$ be a continuous distribution of
$k$-dimensional planes defined everywhere on $\widetilde{\Sigma}_{k}$.
Then $K$ projects onto a continuous field $Q$ of $k$-planes on
$S^{n-1}$.\end{lem} \begin{proof} Let $C\subseteq\widetilde{M}$
denote a closed submanifold of $\widetilde{M}$ diffeomorphic to $S^{n-1}$.
For each $x\in C$, let $\nu(x)\in T_{x}\widetilde{M}$ denote the
inward pointing normal to $C$ at $x$. Thus $K(\nu(x))\subseteq\widetilde{E}_{1}^{u}(\nu(x))$.
Then note that given $x\in C$ and $v\in\widetilde{E}_{1}^{s}(x)$,
the map $d_{v}\widetilde{\pi}|\widetilde{E}_{1}^{u}:\widetilde{E}_{1}^{u}(v)\rightarrow v^{\perp}\subseteq T_{x}\widetilde{M}$
is a linear isomorphism. The desired continuous field $Q$ of $k$-planes
is then given by\[
Q(x):=d_{\nu(x)}\widetilde{\pi}\left[K(\nu(x))\right]\subseteq T_{x}C.\]

\end{proof} For the case where $M$ is odd dimensional we will use
the bundle $P_{\rho_{0}}$ defined by \eqref{P ro complex}, where
this time $\rho_{0}\in\mathbb{C}$ is given by Corollary \ref{thm:hamen2}.
In this case however the argument is initially simpler. Indeed, by
the following topological result, after lifting $P_{\rho_{0}}$ to
a continuous distribution on $\widetilde{E}_{1}^{u}$, we see immediately
that $P_{\rho_{0}}=E_{1}^{u}$. \begin{thm} For $n$ odd, $S^{n-1}$
admits no $k$-plane distribution for $1\leq k\leq n-2$. \end{thm}
See for instance \cite[Theorem 27.18]{Steenrod1951} for a proof.
\begin{rem} We could alternatively deduce the same result by observing
as before that the space of leaves $\mathcal{L}^{+}:=\widetilde{\Sigma}_{k}/\widetilde{\mathcal{W}}^{+}$
is topologically an $\left(n-1\right)$-sphere, and since $P_{\rho_{0}}$
is invariant under the holonomy transport, $P_{\rho_{0}}$ determines
a continuous distribution of $k$-planes on $S^{n-1}$. The proof
given above does \emph{not} use the fact that $P_{\rho_{0}}$ is invariant
under holonomy transport. \end{rem} Unfortunately this does not quite
nail the result as in the even dimensional case. If $\rho_{0}\in\mathbb{R}$,
we deduce $L^{u}=\rho_{0}\mbox{Id}$, and the desired contradiction
follows just as in the even-dimensional case.

If however $\rho_{0}\notin\mathbb{R}$ then more work is required.
Consider the $\phi_{t}^{1}$-invariant almost complex structure $\mathbb{J}$
on $E^{u}$ of class $C^{1}$ defined by\begin{equation}
\mathbb{J}:=\frac{1}{\mbox{Im}\,\rho_{0}}\left(L^{u}-(\mbox{Re}\,\rho_{0})\mbox{Id}\right).\label{eq:acs}\end{equation}
 Lifting $\mathbb{J}$ to an almost complex structure on $\widetilde{E}_{1}^{u}$,
and then using the construction from Lemma \ref{lem:useful lemma from Feres},
we see that $\mathbb{J}$ induces an almost complex structure on $S^{n-1}$.
This immediately implies that $n=3$ or $n=7$, since the only spheres
that admit almost complex structures are $S^{2}$ and $S^{6}$.

It thus remains to eliminate the cases $n=3$ and $n=7$, and we will
tackle these separately.

\subsection*{The case where $n=3$ or $n=7$.}

The first step in the proof of these two special cases is to show
that the existence of a $\phi_{t}^{1}$-invariant almost complex structure
$\mathbb{J}$ on $E^{u}$ forces both $E^{s}$ and $E^{u}$ to be
of class $C^{\infty}$. The next two results are due to Hamenst\"adt;
see Corollary $2.11$ and Corollary $2.12$ of \cite{Hamenstaedt1995}.

\begin{lem} The almost complex structure $\mathbb{J}$ on $E^{u}$
is parallel with respect to the Kanai connection.\end{lem} \begin{proof}
Since $\mathbb{J}$ is $\phi_{t}^{1}$-invariant and $\nabla_{F_{1}}=L_{F_{1}}$,
we certainly have $\nabla_{F_{1}}\mathbb{J}=0$ and thus it suffices
to show that $\nabla_{X_{s}}\mathbb{J}=0$ for all $X_{s}\in\Gamma(E^{s})$
and $\nabla_{X_{u}}\mathbb{J}=0$ for all $X_{u}\in\Gamma(E^{u})$.
We know that $\mathbb{J}$ is invariant under the holonomy maps $H_{x,y}$
(since otherwise the subbundle $P_{\rho_{0}}^{\mathbb{C}}$ from \eqref{P ro complex}
would be of positive dimension - this contradict the fact that an
even dimensional sphere does not have non-trivial subbundles) and
thus as holonomy transport is the same as parallel transport for $\nabla$,
we see that $\nabla_{X_{s}}\mathbb{J}=0$ for all $X_{s}\in\Gamma(E^{s})$.
We can define a new almost complex structure $\mathbb{J}'$ on $E^{s}$
by the equation \[
\omega(X_{u},\mathbb{J}'X_{s})=\omega(\mathbb{J}X_{u},X_{s}).\]
 Then if $Y_{s}\in\Gamma(E^{s})$, applying $Y_{s}$ to the previous
equation and using that $\nabla\omega=0$ shows that \begin{equation}
\omega(X_{u},(\nabla_{Y_{s}}\mathbb{J}')X_{s})=\omega((\nabla_{Y_{s}}\mathbb{J})X_{u},X_{s}),\label{eq:j}\end{equation}
 and thus $\nabla_{Y_{s}}\mathbb{J}'=0$.

But now the point is the following: we could repeat all of what we
have done above but working with $L^{s}$ not $L^{u}$, and thus obtain
an almost complex structure $\widetilde{\mathbb{J}}$ on $E^{s}$.
The same argument would then show that this almost complex structure
$\widetilde{\mathbb{J}}$ is parallel on $E^{u}$, that is, $\nabla_{X_{u}}\widetilde{\mathbb{J}}=0$
for $X_{u}\in\Gamma(E^{u})$. But straight from the definition, it
is clear that $\widetilde{\mathbb{J}}=\mathbb{J}'$, and thus $\mathbb{J}'$
is parallel on $E^{u}$ as well. But then taking $Y_{u}\in\Gamma(E^{u})$
and plugging it into \eqref{eq:j}, we see that $\nabla_{Y_{u}}\mathbb{J}=0$,
and this completes the proof.\end{proof}

\begin{thm}\label{strong}Suppose the strong unstable bundle $E_{1}^{u}$
admits an almost complex structure $\mathbb{J}$ defined as in \eqref{eq:acs}.
Then $\Sigma_{k}$ admits a real analytic structure, for which the
strong (un)stable bundles $E_{1}^{s}$ and $E_{1}^{u}$ are both real
analytic. Moreover, this real analytic structure is $C^{1}$ diffeomorphic
to the underlying smooth structure of $\Sigma_{k}$.\end{thm}

Before starting the proof, let us recall the following facts about
the spaces of leaves. Recall that the space $\mathcal{L}^{+}$ and
$\mathcal{L}^{-}$ of leaves are defined to be the quotient spaces
$\mathcal{L}^{+}:=\widetilde{\Sigma}_{k}/\sim_{+}$ and $\mathcal{L}^{-}:=\widetilde{\Sigma}_{k}/\sim_{-}$,
where $\sim_{\pm}$ are the equivalence relations on $\widetilde{\Sigma}_{k}$
defined by $v\sim_{\pm}w$ if and only if $w\in\widetilde{W}^{\pm}(v)$
respectively. We let $p_{\pm}:\widetilde{\Sigma}_{k}\rightarrow\mathcal{L}^{\pm}$
denote the quotient maps, and give $\mathcal{L}^{\pm}$ the quotient
topology induced by $p_{\pm}$. For each $x\in\widetilde{M}$ the
restriction $p_{+}|\widetilde{\Sigma}_{k}\cap T_{x}\widetilde{M}$
and $p_{-}|\widetilde{\Sigma}_{k}\cap T_{x}\widetilde{M}$ are homeomorphisms
onto $\mathcal{L}^{+}$ and $\mathcal{L}^{-}$ respectively; thus
$\mathcal{L}^{+}$ and $\mathcal{L}^{-}$ both admit the structure
of topological manifolds, and are homeomorphic to $S^{n-1}$.

The key fact that we will need is that there exists a homeomorphism
$F:\mathcal{L}^{+}\rightarrow\mathcal{L}^{-}$ with the following
property: given any leaves $\widetilde{W}^{+}\in\mathcal{L}^{+}$
and $\widetilde{W}^{-}\in\mathcal{L}^{-}$, the intersection $\widetilde{W}^{+}\cap\widetilde{W}^{-}$
contains a unique flow line of $\widetilde{\phi}_{t}^{1}$ unless
$\widetilde{W}^{-}=F(\widetilde{W}^{+})$, in which case $\widetilde{W}^{+}\cap\widetilde{W}^{-}=\emptyset$.
We will call $F$ a \emph{flip map}.

To prove this we take advantage of the fact that $\widetilde{M}$
admits a metric $g_{0}$ of negative curvature again. As before let
$S^{g_{0}}\widetilde{M}$ denote the unit sphere bundle of $(\widetilde{M},g_{0})$
and $\psi_{t}:S^{g_{0}}\widetilde{M}\rightarrow S^{g_{0}}\widetilde{M}$
the geodesic flow of $(\widetilde{M},g_{0})$. Let $\mathcal{L}_{g_{0}}^{+}$
and $\mathcal{L}_{g_{0}}^{-}$ denote the space of stable and unstable
leaves determined by $\psi_{t}$. Defining similar equivalence relations
$\sim_{+}^{g_{0}}$ and $\sim_{-}^{g_{0}}$ on $S^{g_{0}}\widetilde{M}$,
we can realise $\mathcal{L}_{g_{0}}^{+}$ and $\mathcal{L}_{g_{0}}^{-}$
as quotients of $S^{g_{0}}\widetilde{M}$; let $p_{+}^{g_{0}}:S^{g_{0}}\widetilde{M}\rightarrow\mathcal{L}_{g_{0}}^{+}$
and $p_{-}^{g_{0}}:S^{g_{0}}\widetilde{M}\rightarrow\mathcal{L}_{g_{0}}^{-}$
denote the projections. Theorem \ref{thm:north south dynamic} ensures
we have a homeomorphism $G=G_{\tau}:\widetilde{\Sigma}_{k}\rightarrow S^{g_{0}}\widetilde{M}$
that conjugates the orbits of $\widetilde{\phi}_{t}^{1}$ and $\psi_{t}$,
and this conjugacy then induces maps $G_{+}:\mathcal{L}^{+}\rightarrow\mathcal{L}_{g_{0}}^{+}$
and $G_{-}:\mathcal{L}^{-}\rightarrow\mathcal{L}_{g_{0}}^{-}$. $ $It
is well known that the map $S^{g_{0}}\widetilde{M}\rightarrow S^{g_{0}}\widetilde{M}$
sending $v\mapsto-v$ induces a flip map $F_{0}:\mathcal{L}_{g_{0}}^{+}\rightarrow\mathcal{L}_{g_{0}}^{-}$
for the flow $\psi_{t}$, and thus if $F:=G_{-}^{-1}\circ F_{0}\circ G_{+}$
then $F:\mathcal{L}^{+}\rightarrow\mathcal{L}^{-}$ is a flip map
for $\widetilde{\phi}_{t}^{1}$.

Among other things, the existence of a flip map gives us another way
to view the holonomy transport. Given $v\in\widetilde{\Sigma}_{k}$,
the restriction $p_{+}|\widetilde{W}^{u}(v)$ is a homeomorphism onto
the set $\mathcal{L}^{+}\backslash F(\widetilde{W}^{+}(v))$. Thus
there is a unique map $p_{+}^{v}:\widetilde{\Sigma}_{k}\backslash F(\widetilde{W}^{+}(v))\rightarrow\widetilde{W}^{u}(v)$
such that $p_{+}\circ p_{+}^{v}=p_{+}$. Given $w\in\widetilde{W}^{+}(v)$,
if \begin{equation}
V:=\widetilde{W}^{u}(w)\backslash F(\widetilde{W}^{+}(w))\label{eq:def of v}\end{equation}
 then it is easy to see that $p_{+}^{v}|V$ is precisely the restriction
to $V$ of the holonomy map $\widetilde{\mathcal{H}}_{w,v}^{\gamma}:\widetilde{W}^{u}(w)\rightarrow\widetilde{W}^{u}(v)$.
We will use this fact in the proof below.

\begin{proof}\emph{(of Theorem \ref{strong})}

Fix $v\in\Sigma_{k}$, and suppose $X$ and $Y$ are parallel sections
of $E^{u}$ over $W^{u}(v)$. Then $[X,Y]=0$ as $\nabla$ is torsion
free when restricted to $W^{u}(v)$. Since $\mathbb{J}$ is parallel
by the previous lemma, so are $\mathbb{J}X$ and $\mathbb{J}Y$; hence
$[\mathbb{J}X,\mathbb{J}Y]=[\mathbb{J}X,Y]=[X,\mathbb{J}Y]=0$; in
other words the Nijenhuis tensor $N_{\mathbb{J}}$ of $\mathbb{J}$
vanishes, and consequently $\mathbb{J}$ is integrable (see for instance,
\cite[Chapter IX, Theorem 2.4]{KobayashiNomizu1963a}). Thus we obtain
a complex structure on $W^{u}(v)$. Using the almost complex structure
$\mathbb{J}'$ defined in the previous lemma (which we shall now also
refer to simply as $\mathbb{J}$) and repeating the same argument
gives us complex structures on the stable manifolds $W^{s}(v)$.

Let us now pass to the universal cover. We have shown that each of
the strong stable and unstable leaves $\widetilde{W}^{s}$ and $\widetilde{W}^{u}$
admit $\pi_{1}(M)$-equivariant complex structures. We will use this
to define the structure of a complex manifold on the spaces $\mathcal{L}^{+}$
and $\mathcal{L}^{-}$ of leaves. We have shown above that given $v\in\widetilde{\Sigma}_{k}$,
the map \[
p_{+}|\widetilde{W}^{u}(v):\widetilde{W}^{u}(v)\rightarrow U(v):=\mathcal{L}^{+}\backslash F(\widetilde{W}^{+}(v))\]
 is a homeomorphism. Let $\varphi_{v}:U(v)\rightarrow\widetilde{W}^{u}(v)\cong\mathbb{C}^{m}$
(here $n=2m+1$) denote the inverse map. We want to define the structure
of a complex manifold on $\mathcal{L}^{+}$ by taking $\{(\varphi_{v},U(v)\,:\, v\in\widetilde{\Sigma}_{k}\}$
to be an atlas. The fact that the overlap maps $\varphi_{v}\circ\varphi_{w}^{-1}$
are holomorphic where defined follows immediately from the fact that
if $v\in\widetilde{\Sigma}_{k}$, $w\in\widetilde{W}^{+}(v)$, and
$V\subseteq\widetilde{W}^{u}(v)$ is as defined above in \eqref{eq:def of v}
then \[
d_{w}(p_{+}^{v}|V)=\widetilde{H}_{w,v}|V,\]
 and thus $p_{+}^{v}|V$ is holomorphic, since the lifted complex
structure $\widetilde{\mathbb{J}}$ is invariant under the holonomy
transport maps.

Similarly $\mathcal{L}^{-}$ admits the structure of a complex manifold.
Note that by construction the complex structure on $\mathcal{L}^{+}$
and $\mathcal{L}^{-}$ is $\pi_{1}(M)$-equivariant. Next, we claim
that the orbit space $\widetilde{\Sigma}_{k}/\widetilde{\phi}_{t}^{1}$
is homeomorphic to $\mathcal{L}^{+}\times\mathcal{L}^{-}\backslash K$
where $K$ is a closed set. Set \[
\Delta:=\left\{ (p_{+}^{g_{0}}(v),p_{-}^{g_{0}}(v))\in\mathcal{L}_{g_{0}}^{+}\times\mathcal{L}_{g_{0}}^{-}\,:\, v\in S^{g_{0}}\widetilde{M}\right\} .\]
 The following is well known: given a flow line $C_{v}(t)=\psi_{t}v$,
the map \[
C_{v}\mapsto(p_{+}^{g_{0}}(v),F_{0}(p_{-}^{g_{0}}(v)))\in\mathcal{L}_{g_{0}}^{+}\times\mathcal{L}_{g_{0}}^{-}\]
 defines a homeomorphism between the orbit space $S^{g_{0}}\widetilde{M}/\psi_{t}$
of $\psi_{t}$ and $\mathcal{L}_{g_{0}}^{+}\times\mathcal{L}_{g_{0}}^{-}\backslash\Delta$.
It follows that if \[
K:=\left\{ (p_{+}(v),p_{-}(v))\in\mathcal{L}^{+}\times\mathcal{L}^{-}\,:\, v\in\widetilde{\Sigma}_{k}\right\} ,\]
 then given a flow line $\Gamma_{v}(t)=\widetilde{\phi}_{t}^{1}v$,
the map \[
\Gamma_{v}\mapsto(p_{+}(v),F(p_{-}(v)))\in\mathcal{L}^{+}\times\mathcal{L}^{-}\]
 defines a homeomorphism from $\widetilde{\Sigma}_{k}/\widetilde{\phi}_{t}^{1}$
to $\mathcal{L}^{+}\times\mathcal{L}^{-}\backslash K$.

This essentially completes the proof. Indeed, we have shown that $\mathcal{L}^{+}$
and $\mathcal{L}^{-}$ carry $\pi_{1}(M)$-equivariant complex structures,
that the orbit space $\widetilde{\Sigma}_{k}/\widetilde{\phi}_{t}^{1}$
is $C^{1}$ equivalent to the complex manifold $\mathcal{L}^{+}\times\mathcal{L}^{-}\backslash K$
provided with the product complex structure, and that this correspondence
is equivariant under the action of $\pi_{1}(M)$. Thus $\Sigma_{k}$
carries a real analytic structure for which the bundles $E^{+}$ and
$E^{-}$ are real analytic, which by construction is $C^{1}$ diffeomorphic
to the underlying smooth structure on $\Sigma_{k}$.\end{proof}

Write $\rho_{0}=\sigma+it$, so that $\mathbb{J}=\sigma^{-1}(L^{u}-t\mbox{Id})$.
Thus $L^{u}=\sigma J+t\mbox{Id}$, and thus if $K:=\left|\rho_{0}\right|^{-2}\cdot(-\sigma\mathbb{J}+t\mbox{Id})$
then $L^{u}\circ K=K\circ L^{u}=\mbox{Id}$. Then since $d\lambda(\cdot,\cdot)=\omega_{1}(L^{s}(\cdot),\cdot)$
and $\omega_{1}^{n-1}\ne0$, it follows that $(d\lambda)^{n-1}\ne0$
and thus $\lambda\wedge(d\lambda)^{n-1}$ is a volume form.

Now we quote the following theorem of Benoist, Foulon and Labourie
(\cite{BenoistFoulonLabourie1992}, Theorem $1$). \begin{thm} Let
$N$ be a closed manifold of odd dimension $2n-1$ with $n\geq3$
an odd integer. Let $\phi_{t}:N\rightarrow N$ be an Anosov flow with
infinitesimal generator $F$. Suppose that the strong bundles $E^{s}$
and $E^{u}$ are of class $C^{\infty}$. Define a $1$-form $\lambda$
by $\ker\,\lambda=E^{s}\oplus E^{u}$ and $\lambda(F)=1$ and suppose
$\lambda\wedge(d\lambda)^{n-1}$ is a volume form. Then there exists
a unique cohomology class $\eta\in H^{1}(N,\mathbb{R})$ and a closed
$1$-form $\beta$ representing $\eta$ such that $1+\beta(F)>0$
and such that if $Y:=(1+\beta(F))^{-1}F$ then the flow of $Y$ is
$C^{\infty}$ conjugate to the geodesic flow of a Riemannian manifold
of constant negative curvature. \end{thm} Thus we conclude that there
exists a Riemannian manifold $(N,h_{0})$ of constant negative curvature,
a $C^{1}$-diffeomorphism $G:SN\rightarrow\Sigma_{k}$ and a constant
$c>0$ such that \begin{equation}
\rho_{t}\circ G=G\circ\tau_{ct},\label{eq:conjugacy}\end{equation}
 where $\rho_{t}$ is the flow of the vector field $Y=(1+\beta(F_{1}))^{-1}F_{1}$,
where $\beta\in\Omega^{1}(\Sigma_{k})$ is the closed $1$-form given
by the theorem above, and $\tau_{t}:SN\rightarrow SN$ is the geodesic
flow of $N$. In fact, by the main result in \cite{Hamenstaedt1993},
$G$ is of class $C^{2}$.

\begin{rem}Actually for the case $n=3$ we could bypass the above
and conclude immediately that $\phi_{t}^{1}$ is $C^{\infty}$-orbit
equivalent to the geodesic flow of a closed three-dimensional hyperbolic
manifold using recent work of Fang \cite{Fang2007}. Namely, for $n=3$
(so $\dim\, E^{u}=2$), an almost complex structure is the `same'
as having a conformal structure, and thus the fact that $E^{u}$ admits
a $\phi_{t}^{1}$-invariant almost complex structure $\mathbb{J}$
implies that $\phi_{t}^{1}$ is \emph{quasiconformal} (see \cite{Fang2007}).
Then \cite[Theorem 3]{Fang2007} tells us that up to finite covers,
$\phi_{t}^{1}$ is $C^{\infty}$-orbit equivalent to the suspension
of a symplectic hyperbolic automorphism or to the geodesic flow of
a closed three-dimensional hyperbolic manifold. The former is impossible
by Theorem \ref{thm:north south dynamic}. This method does not appear
to work for $n=7$ however.\end{rem}

Let $\Omega_{\textrm{inv}}^{2}(\Sigma_{k})$ denote the set of $2$-forms
on $\Sigma_{k}$ that are invariant under $\phi_{t}^{1}$. It is easy
to see that $2$-forms invariant under $\phi_{t}^{1}$ are precisely
the same as the $2$-forms invariant under $\rho_{t}$.

Similarly let $\Omega_{\textrm{inv}}^{2}(SN)$ denote the set of $2$-forms
on $SN$ that are invariant under $\tau_{t}$. Note that $\omega_{1}$
and $d\lambda$ both lie in $\Omega_{\textrm{inv}}^{2}(\Sigma_{k})$,
and hence by \eqref{eq:conjugacy} the $2$-forms $G^{*}\omega_{1}$
and $G^{*}d\lambda$ on $SN$ both lie in $\Omega_{\textrm{inv}}^{2}(SN)$.\\

The following result is due to Kanai (\cite[Claim 3.3]{Kanai1993}).
\begin{thm} \label{pro:kanai prop}If $n\geq3$ then $\Omega_{\textrm{\emph{inv}}}^{2}(SN)$
is $1$-dimensional (where $\dim\, N=2n-1$), spanned by the canonical
symplectic form $\omega_{N}$ on $TN$. If $n=2$ then $\Omega_{\textrm{\emph{inv}}}^{2}(SN)$
is $2$-dimensional, spanned by $\omega_{N}$ and another $2$-form
$\psi_{N}$. This form can be uniquely characterized as follows. Let
$\Pi:S\mathbb{H}^{3}\rightarrow SN$ denote the universal covering
of $SN$, where $\mathbb{H}^{3}$ denotes hyperbolic $3$-space. Then
there exists a $2$-form $\psi\in\Omega^{2}(S\mathbb{H}^{3})$ which
we will define below, and $\psi_{N}$ is then the unique $2$-form
on $SN$ such that $p^{*}\psi_{N}=\psi$. \end{thm} The theorem immediately
implies the result for the case $n=7$. Indeed, we deduce that $G^{*}\omega_{1}=c_{1}\omega_{N}$
and $G^{*}d\lambda=c_{2}\omega_{N}$ for two constants $c_{1},c_{2}\in\mathbb{R}$
with $c_{1},c_{2}\ne0$, and from this it is clear that $d\lambda=\rho\omega_{1}$
for some constant non-zero $\rho\in\mathbb{R}$. The case $n=3$ still
requires a little work though.\\

Now we define the $2$-form $\psi$ on $S\mathbb{H}^{3}$ mentioned
in the statement of Theorem \ref{pro:kanai prop}. The Jacobi equation
for $\mathbb{H}^{3}$ is given by \[
\ddot{J}-J=0,\]
 (where the dot denotes the covariant derivative along $\gamma$)
and thus it is easy to see that given a geodesic $\gamma$, the normal
Jacobi fields along $\gamma$ are linear combinations of the fields
\[
J(t)=e^{\pm t}U(t),\]
 where $U(t)$ is any parallel normal vector field along $\gamma$.
Let $\phi_{t}^{\mathbb{H}}$ denote the geodesic flow on $\mathbb{H}^{3}$.
Then $\phi_{t}^{\mathbb{H}}$ is Anosov (see Example \ref{exa:example sec curv}).
Given $v\in S\mathbb{H}^{3}$, let $\gamma_{v}$ denote the unique
geodesic adapted to $v$. Then, given $\xi\in T_{v}S\mathbb{H}^{3}$,
let $J_{\xi}$ denote the unique Jacobi field along $\gamma_{v}$
with $J_{\xi}(0)=d_{v}\pi_{\mathbb{H}}(\xi)$ and $\dot{J}_{\xi}(0)=K_{\mathbb{H},v}(\xi)$
(here $\pi_{\mathbb{H}}:S\mathbb{H}^{3}\rightarrow\mathbb{H}^{3}$
and $K_{\mathbb{H}}:TT\mathbb{H}^{3}\rightarrow T\mathbb{H}^{3}$
denote the footpoint and connection maps of $\mathbb{H}^{3}$ respectively.

Using the fact that $d\phi_{t}^{\mathbb{H}}(\xi)=(J_{\xi}(t),\dot{J}_{\xi}(t))$
(see for instance \cite[Section 1.5]{Paternain1999}), it easily follows
that the Anosov splitting $S\mathbb{H}^{3}=\mathbb{R}F_{\mathbb{H}}\oplus E_{\mathbb{H}}^{s}\oplus E_{\mathbb{H}}^{u}$
is given by \[
E_{\mathbb{H}}^{s}(v)=\left\{ \xi\in T_{v}S\mathbb{H}^{3}\,:\,\xi_{H}=-\xi_{V}\right\} ,\]
 and similarly \[
E_{\mathbb{H}}^{u}(v)=\left\{ \xi\in T_{v}S\mathbb{H}^{3}\,:\,\xi_{H}=\xi_{V}\right\} ,\]
 \[
\mathbb{R}F_{\mathbb{H}}(v)=\left\{ \xi\in T_{v}S\mathbb{H}^{3}\,:\,\xi_{V}=0,\xi_{H}=av,a\in\mathbb{R}\right\} .\]
 Using this decomposition we can define an almost complex structure
on the subbundle $E_{\mathbb{H}}=E_{\mathbb{H}}^{s}\oplus E_{\mathbb{H}}^{u}$.
Indeed, fix $v\in S\mathbb{H}^{3}$ and note that the isomorphism
\[
T_{v}T\mathbb{H}^{3}\cong T_{x}\mathbb{H}^{3}\oplus T_{x}\mathbb{H}^{3}\]
 described at the beginning of Section \ref{sec:North-South-Dynamics}
restricts to define an isomorphism\[
E_{\mathbb{H}}(v)\cong v^{\perp}\oplus v^{\perp},\]
 where \[
v^{\perp}:=\left\{ w\in T_{x}\mathbb{H}^{3}\,:\,\left\langle w,v\right\rangle =0\right\} .\]
 Now $v^{\perp}$ is $2$-dimensional, and let $\{e_{1}(v),e_{2}(v)\}\subset v^{\perp}$
be an orthonormal basis such that $\{e_{1}(v),e_{2}(v),v\}$ is a
positively oriented basis of $T_{x}\mathbb{H}^{3}$. This allows us
to define a map $i_{v}:v^{\perp}\rightarrow v^{\perp}$ by $i_{v}e_{1}(v)=e_{2}(v)$
and $i_{v}e_{2}(v)=-e_{1}(v)$.

Now define $\mathbb{J}_{v}:E_{\mathbb{H}}^{s}(v)\rightarrow E_{\mathbb{H}}^{s}(v)$
by \[
\mathbb{J}_{v}(\xi_{H},-\xi_{H})=(i_{v}\xi_{H},-i_{v}\xi_{H}),\]
 and define $\mathbb{J}_{v}:E_{\mathbb{H}}^{u}(v)\rightarrow E_{\mathbb{H}}^{u}(v)$
by \[
\mathbb{J}_{v}(\xi_{H},\xi_{H})=(-i_{v}\xi_{H},-i_{v}\xi_{H}).\]
 This defines an almost complex structure on $E_{\mathbb{H}}$. For
convenience, extend $\mathbb{J}$ to all of $TS\mathbb{H}^{3}$ by
letting $\mathbb{J}|\mathbb{R}F_{\mathbb{H}}\equiv0$.

Finally, we define $\psi\in\Omega^{2}(S\mathbb{H}^{3})$ by \[
\psi_{v}(\xi,\eta)=\omega_{\mathbb{H}}|_{v}(\mathbb{J}_{v}\xi,\eta)\ \ \ \mbox{for\ \ \ }\xi,\eta\in T_{v}S\mathbb{H}^{3}.\]
 This form $\psi$ is the $2$-form referred in the statement of Theorem
\ref{pro:kanai prop}. Note that $\psi$ is closed because $d\psi$
is an invariant 3-form which must vanish by the proof of Lemma \ref{lem:.dlambda is parallel}.\\

The sphere bundle $S\mathbb{H}^{3}$ is trivial: $S\mathbb{H}^{3}=\mathbb{H}^{3}\times S^{2}$.
Given $x\in\mathbb{H}^{3}$, let $S_{x}$ be the 2-sphere of unit
vectors at $x$. Let us make the following observation which shows
that $[\psi]\in H^{2}(S\mathbb{H}^{3},\mathbb{R})=\mathbb{R}$ is
not zero. \begin{lem} \label{lem:The-class- not trivial}For all
$x\in\mathbb{H}^{3}$ we have \[
\int_{S_{x}}\psi\ne0.\]
 \end{lem} \begin{proof} Fix $x\in\mathbb{H}^{3}$ and $v\in S_{x}\mathbb{H}^{3}$.
Let us take two vectors $\xi,\eta\in V(v)$, with say, \[
\xi=(0,\xi_{V})\in T_{x}\mathbb{H}^{3}\oplus T_{x}\mathbb{H}^{3},\ \ \ \eta=(0,\eta_{V})\in T_{x}\mathbb{H}^{3}\oplus T_{x}\mathbb{H}^{3}.\]
 It suffices to observe that $\psi_{v}(\xi,\eta)\ne0$ if $\xi$ and
$\eta$ are not colinear. We need to express $\xi$ as a sum of elements
of $E_{\mathbb{H}}^{s}(v)$ and $E_{\mathbb{H}}^{u}(v)$: this is
easily done by noting\[
\xi=\underset{\in E_{\mathbb{H}}^{s}(v)}{\underbrace{\frac{1}{2}(-\xi_{V},\xi_{V})}}+\underset{\in E_{\mathbb{H}}^{u}(v)}{\underbrace{\frac{1}{2}(\xi_{V},\xi_{V})}}.\]
 Then we have\[
\mathbb{J}_{v}(\xi)=\frac{1}{2}(-i_{v}\xi_{V},+i_{v}\xi_{V})+\frac{1}{2}(-i_{v}\xi_{V},-i_{v}\xi_{V})=-(i_{v}\xi_{V},0).\]
 Thus:\begin{eqnarray*}
\psi_{v}(\xi,\eta) & = & \omega_{\mathbb{H}}|_{v}(\mathbb{J}_{v}\xi,\eta)\\
 & = & -\left\langle i_{v}\xi_{V},\eta_{V}\right\rangle ,\end{eqnarray*}
 and this is non-zero if $\xi$ and $\eta$ are not colinear. \end{proof}
From this it is now easy to complete the proof in the case $n=3$.
Since both $\omega_{0}$ and $d\lambda$ are exact (recall $\omega_{0}=-d\alpha$;
see \eqref{eq:alpha}) it follows that both $\Pi^{*}G^{*}\omega_{1}$
and $\Pi^{*}G^{*}d\lambda$ vanish in $H^{2}(S\mathbb{H}^{3},\mathbb{R})$
(since $\omega_{1}$ is equal to $\omega_{0}$ plus the pullback of
a form on the base), and thus they must both be multiples of $\omega_{\mathbb{H}}$;
that is, \[
\Pi^{*}G^{*}\omega_{1}=c_{1}\omega_{\mathbb{H}},\ \ \ \Pi^{*}G^{*}d\lambda=c_{2}\omega_{\mathbb{H}}\,:\, c_{1},c_{2}\in\mathbb{R}\backslash\{0\}.\]
 Hence again $d\lambda=\rho\omega_{1}$ for some non-zero $\rho\in\mathbb{R}$,
which completes the proof for the case $n=3$, and thus finally completes
the proof of Theorem A.

\subsection*{\label{sec:Twisted-Geodesic-Flows}Proof of the Corollary B}

Now we prove Corollary B from the introduction. \begin{prop} \label{pro:ANOSOV 12PINCHED}Suppose
$g$ is negatively curved and strictly $1/4$-pinched. For sufficiently
large $k$, the Hamiltonian structure $\left(\Sigma_{k},\omega_{1}\right)$
is Anosov and satisfies the $1/2$-pinching condition.\end{prop}
\begin{proof} First of all we will show that $\left(\Sigma_{k},\omega_{1}\right)$
is an Anosov Hamiltonian structure for $k$ sufficiently large. It
is well known that $\phi_{t}^{0}:SM\rightarrow SM$ is Anosov. By
structural stability of Anosov flows, there exists $\delta_{0}>0$
such that $\phi_{t}^{\varepsilon}:SM\rightarrow SM$ is Anosov for
all $\varepsilon\in(-\delta_{0},\delta_{0})$ (see for instance \cite[Corollary 18.2.2]{KatokHasselblatt1995}).

Consider now the map \[
h_{\varepsilon}:TM\rightarrow TM\,,\, v\mapsto\varepsilon v.\]
 Then the observation is that \[
h_{\varepsilon}^{*}\omega_{\varepsilon}=\varepsilon\omega_{1},\ \ \ h_{\varepsilon}^{*}H=\varepsilon^{2}H.\]
 Now $\phi_{t}^{\varepsilon}:SM\rightarrow SM$ is Anosov if and only
if the flow (temporarily called) $\psi_{t}$ of $h_{\varepsilon}^{*}H$
with respect to $h_{\varepsilon}^{*}\omega_{\varepsilon}$ is Anosov
on $(h_{\varepsilon}^{*}H)^{-1}(1/2)$. But $\psi_{t}$ is the Hamiltonian
flow of $\left(x,v\right)\mapsto\frac{\varepsilon^{2}}{2}\left|v\right|^{2}$
with respect to the symplectic form $\varepsilon\omega_{1}$. But
it is easy to see that $\psi_{t}=\phi_{\varepsilon t}^{1}$; hence
we have shown that \[
\phi_{t}^{\varepsilon}:SM\rightarrow SM\mbox{ is Anosov }\Leftrightarrow\phi_{t}^{1}:\Sigma_{1/2\varepsilon^{2}}\rightarrow\Sigma_{1/2\varepsilon^{2}}\mbox{ is Anosov.}\]
 Thus in particular for \[
k>\frac{1}{2}\delta_{0}^{-2},\]
 $\left(\Sigma_{k},\omega_{1}\right)$ is an Anosov Hamiltonian structure.

In order to prove that $\left(\Sigma_{k},\omega_{1}\right)$ is $1/2$-pinched
for $k$ large we note that an equivalent claim is that $(\Sigma_{1/2},\omega_{\varepsilon\sigma})$
is $1/2$-pinched for small $\varepsilon$. An inspection of the proof
of (\ref{eq:a1}) and (\ref{eq:a2}) in \cite[Theorem 3.2.17]{Klingenberg1995}
(see Example \ref{exa:example sec curv}) shows that if we do the
same analysis for the magnetic Jacobi (or Riccati) equation we obtain
numbers $k_{1}(\varepsilon)$ and $k_{0}(\varepsilon)$ for which
(\ref{eq:a1}) and (\ref{eq:a2}) hold. These numbers will be as close
as we wish to $k_{1}(0)=2$ and $k_{0}(0)=\sqrt{-\max\, K}>1$ if
$\varepsilon$ is small enough and the $1/2$-pinching condition follows
(see Example \ref{exa:example sec curv} again).

\end{proof} From this it is easy to prove Corollary B. Suppose $k$
is chosen large enough such that $(\Sigma_{k},\omega_{1})$ is an
Anosov Hamiltonian structure satisfying the $1/2$-pinching condition,
and suppose for contradiction that $(\Sigma_{k},\omega_{1})$ is stable,
and let $\lambda$ be a stabilizing $1$-form. By Theorem A, $(\Sigma_{k},\omega_{1})$
is contact, that is, $d\lambda=\rho\omega_{1}$ for some non-zero
$\rho\in\mathbb{R}$. In particular, $\omega_{1}$ is exact which
in turn implies that $\sigma$ is exact.

\bibliographystyle{amsplain}

\bibliographystyle{amsplain}
\bibliography{C:/Users/Will/Documents/Lyx/willbibtex}

\end{document}